\RequirePackage{fix-cm}
\RequirePackage{amsmath}
\RequirePackage[hyphens]{url}
\documentclass[smallextended]{svjour3hack}       
\providecommand{\openbox}{\leavevmode
  \hbox to.77778em{%
  \hfil\vrule
  \vbox to.575em{\hrule width.5em\vfil\hrule}%
  \vrule\hfil}}
\makeatletter
\DeclareRobustCommand{\qed}{%
  \ifmmode
    \eqno \def\@badmath{$$}
    \let\eqno\relax \let\leqno\relax \let\veqno\relax
    \hbox{\openbox}%
  \else
    \leavevmode\unskip\penalty9999 \hbox{}\nobreak\hfill
    \quad\hbox{\openbox}%
  \fi
}
\makeatother

\usepackage{graphicx}
\usepackage{amssymb}
\usepackage{placeins}
\usepackage{xcolor}
\usepackage[caption=false]{subfig}
\usepackage{multirow}
\usepackage[page]{appendix} 

\usepackage[hidelinks,hypertexnames=false,colorlinks=true,breaklinks=true,bookmarks=true,urlcolor=blue,citecolor=blue,linkcolor=blue,bookmarksopen=false,draft=false]{hyperref}

\usepackage[linesnumbered, boxed, english, onelanguage, vlined]{algorithm2e}
\usepackage{mathtools}

\DeclareMathOperator{\ldet}{ldet}
\DeclareMathOperator{\Trace}{Trace}

\newtheorem{jconjecture}{Conjecture}

\makeatletter
\let\c@proposition\c@theorem
\let\c@corollary\c@theorem
\let\c@lemma\c@theorem
\let\c@definition\c@theorem
\let\c@example\c@theorem
\let\c@remark\c@theorem
\let\c@jremark\c@theorem
\let\c@jconjecture\c@theorem
\makeatother

\makeatletter
\renewcommand*{\thetable}{\arabic{table}}
\renewcommand*{\thefigure}{\arabic{figure}}
\let\c@table\c@figure
\makeatother

\vbadness=\maxdimen

\newcommand{\zlinxthing}{\mbox{z}_{\mbox{\protect\tiny L}}}
\newcommand{\zlinx}{\hyperlink{zlinxtarget}{\zlinxthing}}

\newcommand{\znaturalthing}{\mbox{z}_{\mbox{\protect\tiny $\mathcal{N}$}}}
\newcommand{\znatural}{\hyperlink{znaturaltarget}{\znaturalthing}}

\newcommand{\zgammathing}{\mbox{z}_{\mbox{\protect\tiny $\Gamma$}}}
\newcommand{\zgamma}{\hyperlink{zgammatarget}{\zgammathing}}

\newcommand{\znthing}{\mathfrak{z}_{\mbox{\protect\tiny $\mathcal{N}$}}}
\newcommand{\zn}{\hyperlink{zntarget}{\znthing}}

\newcommand{\zsthing}{\mathfrak{z}_{\mbox{\protect\tiny $\mathcal{S}$}}}
\newcommand{\zs}{\hyperlink{zstarget}{\zsthing}}

\newcommand{\zhthing}{\mathfrak{z}_{\mbox{\protect\tiny $\mathcal{H}$}}}
\newcommand{\zh}{\hyperlink{zhtarget}{\zhthing}}

\newcommand{\zgthing}{\mathfrak{z}_{\mbox{\protect\tiny $\Gamma$}}}
\newcommand{\zg}{\hyperlink{zgtarget}{\zgthing}}

\newcommand{\zcgthing}{\mathfrak{z}_{\mbox{\protect\tiny c-$\Gamma$}}}
\newcommand{\zcg}{\hyperlink{zcgtarget}{\zcgthing}}

\newcommand{\zfuthing}{\mathfrak{z}}
\newcommand{\zfu}{\hyperlink{zfutarget}{\zfuthing}}

\DeclareMathOperator{\Diag}{\textstyle{Diag}}
\DeclareMathOperator{\diag}{\textstyle{diag}}

\makeatletter
\renewcommand*{\top}{%
  {\mathpalette\@transpose{}}%
}
\newcommand*{\@transpose}[2]{%
  \raisebox{\depth}{$\m@th#1\scriptscriptstyle\mathsf{T}$}%
}
\makeatother

\DeclareMathOperator{\rank}{rank}

\begin{document}

\title{Branch-and-bound for integer D-Optimality with fast local search and variable-bound tightening\footnote{A short preliminary version of some of this work appeared in  SOBRAPO 2022; see \cite{PonteFampaLeeSBPO22}}}
\titlerunning{B{\&}B for D-Opt with fast local search and bound tightening}  \author{Gabriel Ponte \and Marcia Fampa \and Jon Lee}
\authorrunning{Ponte, Fampa \& Lee}  

\institute{G. Ponte and M. Fampa\at
            Federal University of Rio de Janeiro \\
            \email{pontegabriel@cos.ufrj.br},
               \email{fampa@cos.ufrj.br} 
           \and
           G. Ponte and J. Lee  \at
            University of Michigan \\
              \email{gabponte@umich.edu}, \email{jonxlee@umich.edu}
}

\date{}


\maketitle
\begin{abstract}
We develop a branch-and-bound  algorithm
for the integer D-optimality problem, a central problem in statistical design theory,  
based on two convex relaxations,  employing
 variable-bound tightening and fast local-search procedures, testing our ideas on 
 various test problems. 
\end{abstract}

\section{Introduction}

We consider the integer D-Optimality problem formulated as 
\vskip-10pt 

\begin{align*}\label{prob}\tag{D-Opt}
\textstyle
&\mbox{z}:=\max \left\{ \ldet\textstyle\sum_{\ell\in N} x_\ell v_\ell v_\ell^\top \, : \, \mathbf{e}^\top x=s,~  l\leq x \leq u,~ x\in\mathbb{Z}^n
\right\}\\[4pt]
&\qquad=\max \left\{ \ldet \left(
\textstyle\sum_{\ell\in N} l_\ell v_\ell v_\ell^\top 
+ \textstyle\sum_{\ell\in N} x_\ell v_\ell v_\ell^\top\right) \, : \right.\\[-1pt]
&\qquad\qquad\qquad\qquad\, \left.\mathbf{e}^\top x=s-\mathbf{e}^\top l\,,~  0\leq x \leq u-l,~ x\in\mathbb{Z}^n
\vphantom{\textstyle\sum_{\ell\in N} l_\ell v_\ell v_\ell^\top}\right\},
\end{align*}
where $v_\ell \in \mathbb{R}^{m}$, 
for $\ell\in N:=\{1,\ldots,n\}$,
with $n$ and $s\geq m$
natural numbers, and 
$0\leq l < u\in\mathbb{Z}^n$, with
$\mathbf{e}^\top l \leq s \leq \mathbf{e}^\top u$.
It will be very useful to define  $A:= (v_1, v_2, \dots, v_n)^\top$
(which we always assume has full column rank), and so we have  
$
\sum_{\ell\in N} x_\ell v_\ell v_\ell^\top = A^\top \Diag(x) A.
$
We note that the related \emph{continuous} D-optimality problem (where we have
$x\in \mathbb{R}_+^n$ with typically no upper bounds) is 
also considered in the literature (see \cite{Fedorov,Nikolov,Boyd1998,Sager2013}, for example).

\ref{prob} is a fundamental problem in statistics, in the area of ``experimental designs'' (see \cite{Puk}, for example). 
Ideally we would be considering the least-squares regression 
problem $\min_{\theta\in \mathbb{R}^{m}} \|A_u\theta -b_u\|_2$\,, where 
$A_u\in \mathbb{R}^{\mathbf{e}^\top u\times m}$ has $v_\ell^\top$  repeated $u_\ell$ times, and  
$b_u\in\mathbb{R}^{\mathbf{e}^\top u}$ is a corresponding response vector. 

But we consider a situation where each $v_\ell$
corresponds to a costly experiment, which should be carried out between $l_\ell$ and  $u_\ell$ times.
Overall, we have a budget to carry out a total of $s(\geq m)$ experiments, and so
we model the choices by   $x$ (in \ref{prob}). 
For a given feasible solution $\tilde{x}$, we define $A_{\tilde{x}}$
to be a matrix that has $v_\ell^\top$ repeated $\tilde x_\ell$ times as its rows, with
$b_{\tilde{x}}$ as the associated response vector.
This leads to the (reduced)  least-squares 
problem $\min_{\theta\in \mathbb{R}^m} \|A_{\tilde x}\theta - b_{\tilde{x}}\|_2$\,,
with solution $\hat\theta:=
(A_{\tilde x}^\top A_{\tilde x})^{-1}A_{\tilde x}^\top 
b_{\tilde x}$. 
The squared volume of a standard ellipsoidal confidence region (centered at $\hat\theta$) for the true $\theta$
is inversely proportional to $\det \sum_{\ell\in N}  \tilde{x}_\ell v_\ell v_\ell^\top$\,.
So \ref{prob} corresponds to picking the set of allowable experiments to 
minimize the volume of the confidence region for  $ \theta$.
The criterion of D-optimality was first suggested by 
A. Wald (see \cite{Wald}). The term ``D-optimality'' was coined 
by J. Kiefer (see \cite{Kiefer}). Also see \cite{KW,Draper,Fedorov}, for example. 

There is a large literature on heuristic 
algorithms for \ref{prob} and its variations.
\cite[Chap. 12]{Atkinson} contains a wealth of information on combinatorial local-search heuristics. 
See  \cite{SinghXie_ApproxDopt} and \cite{MadanSinghTantXie} for approximation algorithms with guarantees. 

\cite{Welch} was the first to approach \ref{prob} with an exact branch-and-bound (B\&B) algorithm,
employing a bound based on Hadamard's inequality and another based on continuous relaxation (apparently without using state-of-the-art NLP solvers of that time).  
\cite{KoLeeWayne,KoLeeWayne2} proposed a spectral bound
and analytically compared it with the Hadamard bound;
 also see \cite{LeeLind2019}.
 \cite{Ucinski2007,Sagnol2015,Ahipasaoglu2021} also consider the application of B\&B algorithms to \ref{prob}.
 \cite{Ucinski2007} addresses the case  where the variables are binary, i.e. $l=0$ and $u=\mathbf{e}$ and applies a  simplicial decomposition algorithm to solve the continuous relaxation. 
\cite{Sagnol2015} proposes to reformulate the continuous relaxation  as  a semidefinite program  or a second-order
 cone program. \cite{Ahipasaoglu2021} proposes a first-order method to solve a continuous relaxation of the problem where the variables are box constrained.
 
\cite{li2022d} applied a local-search procedure and  an exact B\&B algorithm to the Data-Fusion problem: a particular case of the  integer D-optimality problem where  the matrix $\sum_{\ell\in N}l_\ell v_\ell v_\ell^\top $\,, known as the ``existing Fisher Information Matrix (FIM)'', is assumed to be positive definite. Moreover, the Data-Fusion problem  considers only the case where the variables are binary, i.e., $l=0$ and $u=\mathbf{e}$.  Although the Data-Fusion problem and the integer D-optimality problem have similarities, most techniques  used in \cite{li2022d} rely on  the positive definiteness of the existing FIM and cannot be applied to our problem. Closely related to the Data-Fusion problem
is the ``D-optimal design augmentation'' problem (see \cite[Chap. 19]{Atkinson}), but the emphasis of the literature
on that problem is on heuristics, rather than on the relationship of upper bounds, which is the emphasis of our investigation. 

We note that although all of the 
prior work on B\&B algorithms for integer D-optimality, 
that is 
\cite{Welch,KoLeeWayne,KoLeeWayne2,Ucinski2007,Sagnol2015,Ahipasaoglu2021}, did not have access to  our new ``$\Gamma$-bound'',
and did not use variable tightening based on convex duality. With these tools, we
are able to significantly improve the state-of-the-art for B\&B algorithms for integer D-optimality. 

Finally, we 
would like to mention that a similar B\&B approach has been successfully applied to the related maximum-entropy sampling problem (MESP) (see \cite{AFLW_Using,Anstreicher_BQP_entropy,Kurt_linx,FL2022}), where given the covariance matrix $C$ of a  
Gaussian random $n$-vector, one searches for 
a subset of $s$ random variables that maximizes the ``information'' (measured by Shannon's ``differential entropy'')
(see \cite{SW,CaseltonZidek1984,LeeEnv,FL2022}, for example). Those algorithmic works inspired several aspects of the techniques 
we have developed for \ref{prob}. 

\vspace{0.1in} 
\noindent {\bf Organization.}
In \S2, we present the ``natural bound'' for \ref{prob}
and a procedure to fix variables based on convex duality. 
In \S3, we present a new bound for the binary version of \ref{prob}, the ``$\Gamma$-bound'', and we again describe 
how to fix variables. 
 We note that we can reformulate \ref{prob} as a binary
 version by repeating row $\ell$ of $A$ $u_\ell-l_\ell$ times. Additionally, we demonstrate that the $\Gamma$-bound for the binary version of \ref{prob} with a particular variable fixing is precisely the ``complementary 
 $\Gamma$-bound'' for the  so-called ``Data-Fusion problem''.
 Further, we
 present 
  theoretical results showing some relations between  different bounds for  the Data-Fusion problem, including bounds from the literature.
In \S4, we propose 
three local-search heuristics for \ref{prob}, considering  different neighborhoods of the current point to visit at each iteration; we present
 five algorithms to construct  initial solutions for the local-search procedures;
we describe how to   compute the determinant of a rank-one update of a given matrix, knowing the determinant of the matrix --- this procedure is essential for the successful application of the local-search procedures. 
In \S5, 
we present and analyze different strategies to compute search directions and step sizes for each iteration of the local-search procedures. 
In \S6, we present numerical experiments concerning a 
B\&B algorithm based on our results. 

\vspace{0.1in} 
\noindent {\bf Notation.}
 We let $\mathbb{S}^n$  (resp., $\mathbb{S}^n_+$~, $\mathbb{S}^n_{++}$)
 denote the set of symmetric (resp., positive-semidefinite, positive-definite) matrices of order $n$. 
 We let $\Diag(x)$ denote the $n\times n$ diagonal matrix with diagonal elements given by the components of $x\in \mathbb{R}^n$, and $\diag(X)$ denote the $n$-dimensional vector with elements given by the diagonal elements of $X\in\mathbb{R}^{n\times n}$.
  For $X\in\mathbb{S}^n$, we let $d_i(X)$ be the $i$-th largest element of $\diag(X)$, and  $\lambda_i(X)$ be the $i$-th largest eigenvalue of $X$.
We denote an all-ones  vector
by $\mathbf{e}$, an \hbox{$i$-th} standard unit vector by $\mathbf{e}_i$\,.
For compatible $M_1$ and $M_2$, 
$M_1\bullet M_2:=\Trace(M_1^\top M_2)$ is the matrix dot-product.
For matrix $M$, we denote row $i$ by $M_{i\cdot}$ and
column $j$ by $M_{\cdot j}$\,.

\section{The natural bound and its properties}\label{sec:natural_bound}

We consider the convex continuous relaxation of \ref{prob}, formulated as 
\begin{equation}\label{cont_rel}\tag{$\mathcal{N}_{\mbox{\protect\tiny D-Opt}}$}
\hypertarget{znaturaltarget}{\znaturalthing}:=\max \left\{ \ldet \left(A^\top \Diag(x) A\right) \, : \, \mathbf{e}^\top x=s, \, l\leq x\leq u, \, x\in \mathbb{R}^n\right\},
\end{equation}
where we refer to $\mbox{z}_{\mbox{\protect\tiny $\mathcal{N}$}}$ as the \emph{natural bound} for \ref{prob}.

Using similar techniques as \cite[Sec. 3.3.2-3]{FL2022}, 
we formulate the Lagrangian dual  of \ref{cont_rel}  
and use it for  tightening bounds on the variables (and for variable fixing if  strong enough), based on general principles of convex MINLP.

The Lagrangian dual  of \ref{cont_rel} is
\begin{equation}\label{eq:lag_with_theta}\tag{Du-$\mathcal{N}_{{\mbox{\protect\tiny D-Opt}}}$}
\begin{array}{lll}
&\min &-\ldet \Theta    - \omega^\top l + \nu^\top u + \tau s - {m},\\
&\text{s.t.} 
&\diag(A\Theta A^\top)  + \omega - \nu - \tau\mathbf{e} = 0,\\
&&\Theta \succ 0,\nu \geq 0, \omega \geq 0.
\end{array}
\end{equation}
We show in Thm.  \ref{thm:fix_dopt}, how to  tighten variable bounds  for \ref{prob} based on  knowledge of a lower bound  and a feasible solution for  \ref{eq:lag_with_theta}\,.

\begin{theorem}\label{thm:fix_dopt}
Let%
\vspace{-6pt}
\begin{itemize}
    \item LB be the objective-function value of a feasible solution for  \ref{prob};
    \item $(\hat\Theta,\hat\nu,\hat\omega,\hat\tau)$ be a feasible solution for  \ref{eq:lag_with_theta} with objective value $\hat\zeta$.
\end{itemize}
\vspace{-5pt}
Then, for every optimal solution $x^\star$ for \ref{prob}, we have:
\vspace{-5pt}
\begin{align}
     &x_k^\star \leq l_k + \left\lfloor
  \left(\hat{\zeta}-{LB}\right)/\hat\omega_k
     \right\rfloor ,\quad ~\forall\; k \in N\text{ such that } \hat \omega_k>0,\label{ineq1}\\ 
     &x_k^\star \geq u_k-  \left\lfloor
    \left(\hat{\zeta}-{LB}\right)/\hat\nu_k
     \right\rfloor,\quad ~\forall\; k \in N\text{ such that } \hat \nu_k>0.\label{ineq2}
     \end{align}
 \end{theorem}

Next, we show how to construct a closed-form feasible solution of \ref{eq:lag_with_theta} from a feasible solution $\hat x$ of \ref{cont_rel} such that $A^\top\Diag(\hat x)A\!\in \!\mathbb{S}^m_{++}$\,,
 with the goal of having a small duality gap. 
The constructed solution is useful in applying Thm. \ref{thm:fix_dopt}, especially when we  solve \ref{cont_rel} inexactly.
 We define $\hat\Theta:= (A^\top \Diag(\hat x) A)^{-1}$. The minimum   gap between $\hat x$ in \ref{cont_rel} and feasible solutions $(\hat\Theta,\hat\omega,\hat\nu,\hat\tau)$ of  \ref{eq:lag_with_theta}  is the optimal value of the linear program
\begin{equation}\label{eq:g_theta}\tag{$G(\hat\Theta)$}
\begin{array}{rrl}
&\min & - \omega^\top l + \nu^\top u + \tau s - m,\\
&\text{s.t.} 
& \omega - \nu - \tau\mathbf{e} = -\diag(A\hat\Theta A^\top)  ,\\
&&\nu \geq 0, \omega \geq 0.
\end{array}
\end{equation}
To obtain an optimal solution of  \ref{eq:g_theta},   we consider its  optimality conditions 
\begin{equation}\label{kktnatural}
\begin{array}{l}
    \diag(A\hat \Theta A^\top)  + \omega - \nu - \tau\mathbf{e} = 0,~\nu\geq 0,~\omega\geq 0,\\
    \mathbf{e}^\top x = s,~l\leq x \leq u,\\
    - \omega^\top l + \nu^\top u + \tau s = \diag(A\hat \Theta A^\top)^\top x.
\end{array}
\end{equation}
We consider the permutation $\sigma$ of the indices in $N$, such that $\diag(A\hat\Theta A^\top)_{\sigma(1)} \geq \dots \geq \diag(A\hat\Theta A^\top)_{\sigma(n)}$. 
If $u_{\sigma(1)}+\sum_{i=2}^n l_{\sigma(i)}>s$, we let $\varphi:=0$, otherwise we let $\varphi:=\max\{j\in N: \sum_{i=1}^j u_{\sigma(i)} +  \sum_{i=j+1}^n  l_{\sigma(i)}\leq s\}$. 
We define $P := \{\sigma(1),\dots,\sigma(\varphi)\}$ and $Q :=\{\sigma(\varphi + 2),\dots,\sigma(n)\}$.
Then, we  can verify that the following solution satisfies \eqref{kktnatural}.

\begin{align*}
    &\tilde{x}_\ell:= \begin{cases}
        u_\ell\,,\quad&\text{for }\ell \in P;\\
        l_\ell\,,&\text{for }\ell \in Q;\\
        s-\sum_{\ell\in P}u_\ell -\sum_{\ell\in Q}l_\ell\,,&\text{for }\ell = \sigma(\varphi +1), ~ \text{if  }\varphi<n,
    \end{cases}\\ 
    &\tilde{\tau} :=  \begin{cases}
    \diag(A\hat\Theta A^\top)_{\sigma(\varphi+1)}\,,\quad &\text{if  }\varphi<n;\\
    0,&\text{otherwise},
    \end{cases}\\
    &\tilde{\nu}_\ell := \begin{cases}
         \diag(A\hat\Theta A^\top)_{\ell} - \tilde{\tau},\quad &\text{for }\ell \in P;\\
         0,&\text{otherwise},
    \end{cases}\\
    &\tilde{\omega}_\ell := \begin{cases}
         \tilde{\tau}-\diag(A\hat\Theta A^\top)_{\ell}\,,\quad &\text{for } \ell \in Q;\\
         0,&\text{otherwise}.
    \end{cases}
\end{align*}

\section{The \texorpdfstring{$\Gamma$}{Gamma}-bound and its properties}

Next, we introduce a new convex continuous relaxation of \ref{prob}, based on the so-called $\Gamma$-function, using similar techniques to those used in  \cite{Nikolov,Weijun}.

\subsection{The $\Gamma$-function} 

\begin{lemma}\label{Ni13}{\cite[Lem. 13]{Nikolov}}
 Let $\lambda\in\mathbb{R}_+^n$ with $\lambda_1\geq \lambda_2\geq \cdots\geq \lambda_n$\,, and let $t$ be an integer with  $0<t\leq n$. There exists a unique integer $\iota$, with $0\leq \iota< t$, such that
 \begin{equation}\label{reslemiota_a}
 \lambda_{\iota}>\textstyle\frac{1}{t-\iota}\textstyle\sum_{\ell=\iota+1}^n \lambda_{\ell}\geq \lambda_{\iota +1}~,
 \end{equation}
 with the convention $\lambda_0=+\infty$. \qed
\end{lemma}

Suppose that  $\lambda\in\mathbb{R}^n_+$~, and assume that 
$\lambda_1\geq\lambda_2\geq\cdots\geq\lambda_n$~. Given  integer $t$ with $0<t\leq n$,
let $\iota$ be the unique integer defined by Lem. \ref{Ni13}. We define
\begin{equation*}
\phi_t(\lambda):=\textstyle\sum_{\ell=1}^{\iota} \log\left(\lambda_\ell\right) + (t - \iota)\log\left(\frac{1}{t-{\iota}} \sum_{\ell=\iota+1}^{n}
\lambda_\ell\right),
\end{equation*}
and, for $X\in\mathbb{S}_{+}^n$~, we define the \emph{$\Gamma$-function}
\begin{equation}\label{def:gamma}
\Gamma_t(X):= \phi_t(\lambda(X)).
\end{equation}

\begin{proposition}\label{prop:gammaconc}{ \cite[Sec. 4.1]{Nikolov}}
    The $\Gamma$-function defined in \eqref{def:gamma} is  concave  on $\mathbb{S}_{+}^n$\,. 
\end{proposition}

\begin{lemma}\label{lem:iota_t}
 Let $\lambda\in\mathbb{R}_+^n$ with $\lambda_1\geq \lambda_2\geq \cdots\geq
 \lambda_\delta> \lambda_{\delta+1}=\cdots=
 \lambda_r>\lambda_{r+1}=\cdots
 =\lambda_n=0$.
 Then,%
 \vspace{-5pt}
 \begin{enumerate}
     \item[(a)] For $t=r$, the $\iota$ satisfying
 \eqref{reslemiota_a} is
 precisely $\delta$. 
 \item[(b)] For $r<t\leq n$, the $\iota$ satisfying
 \eqref{reslemiota_a} is
 precisely $r$. 
 \end{enumerate}
\end{lemma}

\begin{proof}
For (a), the result follows  because
\[
\textstyle\frac{1}{r-\delta}\sum_{\ell=\delta+1}^n \lambda_{\ell} = \lambda_{\delta+1} \quad\mbox{and}\qquad \lambda_\delta> \lambda_{\delta+1}~.
\]
For (b), the result follows because 
\[
\textstyle\frac{1}{t-r}\sum_{\ell=r+1}^n \lambda_{\ell} = 0 =\lambda_{r+1} \quad\mbox{and}\qquad \lambda_r> \lambda_{r+1}~. \qed
\]
\end{proof}

\begin{theorem}\label{cor:propphi}
Let $\lambda\in\mathbb{R}_+^n$ with $\lambda_1\geq \lambda_2\geq \cdots\geq
 \lambda_r>\lambda_{r+1}=\cdots
 =\lambda_n=0$. It follows from Lem. \ref{lem:iota_t} that  
 \[
 \begin{array}{ll}
 \phi_t(\lambda)=\sum_{\ell=1}^{t} \log\left(\lambda_\ell\right),& \mbox{  for $t= r$,}\\[2pt]
 \phi_t(\lambda)=-\infty, & \mbox{  for $r< t\leq n$,}
 \end{array}
 \]
 where we use $\log(0)=-\infty$. Finally, we note that, in general, 
 \[
 \textstyle\phi_{t}(\lambda)\neq \sum_{\ell=1}^{t} \log\left(\lambda_\ell\right), \quad \mbox{  for $0<t<r$.}
 \]  
\end{theorem}

\subsection{The $\Gamma$-bound for  the 0/1  D-Optimality problem}

 In this section,  we  consider the 0/1  D-Optimality problem
\begin{align*}\label{prob01}\tag{D-Opt(0/1)}
&\max \left\{\ldet(A^\top \Diag(x) A) \, : \, \mathbf{e}^\top x=s,~   x\in\{0,1\}^n\right\},
\end{align*} 
where $A:= (v_1, v_2, \dots,v_n)^\top\in\mathbb{R}^{n\times m}$ has full column rank, and $m\leq s\leq n$. 

\begin{remark}\label{rem:gammafordopt}
    We can reformulate \ref{prob} as \ref{prob01} by 
    replacing $A$ in \ref{prob01} with $A_{u-l}$\,, that is,
    a matrix with  rows comprising $u_\ell-l_\ell$ repetitions of $v_\ell^\top$\,, for $\ell \in  N$. But of course this may not be computationally effective when $u-l$ has large components. 
\end{remark}

Next, we present the $\Gamma$-bound for \ref{prob01}.
 Let $A = U\Sigma V^\top$  be the real singular-value decomposition  of $A$. Consider a factorization $I_n-UU^{\top} = WW^\top$,
with $W\in \mathbb{R}^{n\times (n-m)}$. 

The \emph{$\Gamma$-bound} for \ref{prob01} is defined as

\begin{equation}\label{gamma_bound} \tag{$\Gamma_{{\mbox{\protect\tiny D-Opt(0/1)}}}$}
\begin{array}{lllll}
\hypertarget{zgammatarget}{\zgammathing}:=&2\textstyle\sum_{i= 1}^m \log(\Sigma_{ii}) + &\max  &\Gamma_{n-s}\left(W^\top\Diag(y)W
\right),\\[2pt]
&&\text{s.t.} 
& \mathbf{e}^\top y = n-s,\\
&&&0\leq y\leq \mathbf{e}.
\end{array}
\end{equation}

We note that $I_n - UU^\top \! \in \mathbb{S}^n_{+}$ 
and has rank $n-m$, so an appropriate $W$ can be  obtained from the spectral decomposition of $I_n-UU^\top$; that is,
$I_n-UU^\top=\sum_{i=1}^{n-m} \tilde\lambda_i \tilde{\nu}_i \tilde{\nu}_i^\top$~,
and $w_i:=\sqrt{\tilde{\lambda}_i}\tilde{\nu}_i$ is the $i$-th column of $W$, $i=1,\ldots,n-m$. Similarly to \cite[Sec. 3.4.5]{FL2022}, the $\Gamma$-bound does not depend on the factorization.

In Thm.  \ref{thm:exactrel} we will establish that \ref{gamma_bound} is an \emph{exact relaxation} of \ref{prob01},
in the sense that for any given   feasible solution $\hat x$ to \ref{prob01} with objective value $\hat z$, there is a 
feasible solution $\hat y$ to   \ref{gamma_bound} given by $\hat y:=\mathbf{e}-\hat x$, with  the same objective value $\hat{z}$. Furthermore, the $\Gamma$-function is concave in $\mathbb{S}_+^n$ (see Prop. \ref{prop:gammaconc}). Then, we conclude that \ref{gamma_bound} is a convex exact relaxation of \ref{prob01}.

\begin{theorem}\label{thm:exactrel}
   \ref{gamma_bound} is an exact convex relaxation of \ref{prob01}.
\end{theorem}
\begin{proof}
We have $A = U\Sigma V^\top$, with $U \in \mathbb{R}^{n \times m}$, $\Sigma \in \mathbb{R}^{m \times m}$ and $V \in \mathbb{R}^{m \times m}$, such  that  $U^\top U=V^\top V = VV^\top = I_m$\,.  
Then, for any   feasible solution $x$ for \ref{prob01}, we have
 \begin{align*}
     \ldet(A^\top \Diag(x) A) &= \ldet(V\Sigma^\top U^\top \Diag(x) U \Sigma V^\top)\\
     &=  \ldet(VV^\top) +\ldet(\Sigma^\top\Sigma) + \ldet( U^\top \Diag(x) U )\\
     &= 2\textstyle\sum_{i= 1}^m \log(\Sigma_{ii}) + \ldet( U^\top \Diag(x) U ).
 \end{align*}
  
   Now, let $S\subset N:=\{1,\ldots,n\}$ be the support  of $x$ with $|S| = s$, $T:=N\setminus S$ and $t:=|T| =  n-s$. Then, we have
\begin{align*}
     &\ldet(U^\top \Diag(x) U) ~=~ \ldet( I_m - U^\top U + U^\top \Diag(x) U)\\
          &\quad =  \ldet( I_m - U^\top \Diag(\mathbf{e}-x) U)
     ~=~\ldet(I_{m} - U_{T\cdot}^\top U_{T\cdot})\\
     &\quad =\ldet(I_{t} - U_{T\cdot}U_{T\cdot}^\top)
     ~=~\ldet\left((I_n - UU^\top)_{T,T}\right)\\
     &\quad =\ldet(W_{ T\cdot} W_{ T\cdot}^\top)
     ~=~ \Gamma_t(W^\top\Diag(\mathbf{e}-x)W),
 \end{align*}
here the fourth equation holds because the non-zero eigenvalues of $U_{T\cdot}^\top U_{T\cdot}$ and $U_{T\cdot}U_{T\cdot}^\top$ are the same, and the last equation comes from the definition of the $\Gamma_t$-function, from Thm. \ref{cor:propphi}, and from the fact that $W^\top\Diag(\mathbf{e}-x)W$ is a rank-$t$ matrix with the same $t$ non-zero eigenvalues as  $W_{T\cdot} W_{ T\cdot}^\top$\,.
 
 Then, for $x\in\{0,1\}^n$, we see that $2\sum_{i= 1}^m \log(\Sigma_{ii}) + \Gamma_{n-s}(W^\top\Diag(\mathbf{e}-x)W)=\ldet(A^\top\Diag(x)A)$. 
     The convexity of the relaxation follows from Prop. \ref{prop:gammaconc}.\qed
\end{proof}

\begin{remark}\label{rem:reducetoMESP}
    For a given  $x\in\{0,1\}^n$ with $\mathbf{e}^\top x =s$, let  $S\subset N:=\{1,\ldots,n\}$ to be the support  of $x$ and $T:=N\setminus S$. Then, we see  from the proof of Thm. \ref{thm:exactrel}  that
    \[
    \ldet(A^\top\Diag(x)A)=2\textstyle\sum_{i= 1}^m \log(\Sigma_{ii}) + \ldet((I_n-UU^\top)_{T,T}),
    \]
where $A=U\Sigma V^\top$ is the real value decomposition of $A$.
Therefore, any instance of \ref{prob01} can be reduced to an instance of MESP, where we search for a  maximum (log-)determinant principal submatrix  of  order $n-s$, from
the input positive-semidefinite matrix $I_n-UU^\top$ of order $n$ and rank $n-m$.   
With the same reasoning, we can see that any instance of MESP, where we search for a  maximum (log-)determinant principal submatrix  of order $s$, from
an input positive-semidefinite order-$n$ symmetric matrix  $C$ of rank $m$ for which all positive eigenvalues  are the same can be reduced to an instance of \ref{prob01}, where we select $n-s$ rows of the input   matrix  $A:=U\in\mathbb{R}^{n\times m}$, where $UU^\top= I_n- \textstyle\frac{1}{\lambda_1(C)}C$ and $U^\top U=I_m$ (i.e., $UU^\top$ is the compact spectral decomposition of $I_n-\textstyle\frac{1}{\lambda_1(C)}C$, and all of its nonzero eigenvalues are $1$). We also note that in case the input matrix $C$  is positive definite, then the corresponding instance of MESP can reduced to an instance of the binary Data-Fusion problem, a special case of \ref{prob01} (see \cite[Thm. 1]{li2022d}).
\end{remark}

\subsection{Duality and variable fixing}\label{sec:dualgamma}

Next, analogously to what we presented for the natural bound in \S\ref{sec:natural_bound}, we use similar techniques as \cite[Sec. 3.4.1-4]{FL2022} to formulate the Lagrangian dual  of \ref{gamma_bound} and use it to fix variables on \ref{prob01}. Although in \ref{gamma_bound}\,,  the lower and upper bounds on the variables are zero and one, we will derive the Lagrangian dual of the more general problem with lower and upper bounds on the variables given respectively by $l$ and $u$, that is, we consider the constraints $a\leq y\leq b$ instead of $0\leq y\leq \mathbf{e}$. The motivation for this, is to derive the technique to fix variables at any subproblem considered during the execution of the B\&B algorithm, when some of the variables may already be fixed. Instead of redefining the problem with less variables in our numerical experiments, we found it more efficient to change the upper bound $b_i$ from one to zero, when variable $i$ is fixed at zero in a subproblem, and similarly,  change the lower bound $a_i$ from zero to one, when variable $i$ is fixed at one. 

Then, more generally considering $a$ and $b$ as the lower and upper bounds on $y$ instead of zero and one,  the Lagrangian dual of \ref{gamma_bound} is
\begin{align}
\ldet(A^\top A)~ \!+\! ~&\!\min\, - \!\!
\textstyle \sum\limits_{i=s-m+1}
^{n-m}\!
\log(\lambda_{i} (\Theta ))     - \omega^\top a + \nu^\top b + \tau (n-s) - (n-s)\nonumber \\
&\text{s.t.} \quad
\diag(W\Theta W^\top)  + \omega - \nu - \tau\mathbf{e} = 0,\label{eq:dual_gamma_dopt_01}\tag{Du-$\Gamma_{{\mbox{\protect\tiny D-Opt(0/1)}}}$}\\
&\qquad\quad \Theta \succ 0,\nu \geq 0, \omega \geq 0.\nonumber
\end{align}

Next, we show how to construct a feasible solution of \ref{eq:dual_gamma_dopt_01} from a feasible solution $\hat y$ of \ref{gamma_bound} such that $\hat r:=\rank(W^\top\Diag(\hat y)W)\geq n-s$,  with the goal of producing a small duality gap. We note that if this condition in $\hat y$ was not satisfied we would have  $\Gamma_{n-s}(W^\top\Diag(y)W) \!=\! -\infty$ (see Thm. \ref{cor:propphi}).
 
We  consider the spectral 
decomposition $W^\top\Diag(\hat{y})W=\sum_{\ell=1}^{n-m} \hat \lambda_\ell \hat u_\ell \hat u_\ell^\top\,,$
with $\hat \lambda_1\geq\hat \lambda_2\geq\cdots\geq \hat \lambda_{\hat r}>\hat \lambda_{\hat{r}+1}=\cdots=\hat \lambda_{n-m}=0$. We  define 
$\hat{\Theta}:=\sum_{\ell=1}^{n-m} {\hat \beta}_\ell \hat{u}_\ell \hat{u}_\ell^\top$\,,
where 
\begin{equation}\label{betaepsilon}
\hat{\beta}_\ell:=\left\{
\begin{array}{ll}
        \textstyle 1/\hat{\lambda}_\ell\,,      
       &\mbox{ for }1\leq \ell\leq \hat{\iota};\\ 
     1/\hat{\delta},&\mbox{ for }\hat{\iota}<\ell\leq \hat{r};\\ 
     (1+\epsilon)/\hat{\delta},&\mbox{ for }\hat{r}<\ell\leq n-m,
\end{array}\right.
\end{equation}
where $\epsilon>0$, $\hat{\iota}$ is the unique integer defined  in Lem. \ref{Ni13} for $\lambda=\hat{\lambda}\in\mathbb{R}_+^{n-m}$ and $t=n-s$; and
$
\hat \delta:=\frac{1}{n-s-\hat \iota}\sum_{\ell=\hat \iota+1}^{n-m}\hat \lambda_\ell
$\,.
From Lem. \ref{Ni13}, we have that $\hat\iota<n-s$. Then, as $n-s\leq \hat{r}$,  we have
\begin{equation*} 
-\textstyle \sum_{\ell=1}^{n-s} \log(\hat{\beta}_{\ell})= \textstyle \sum_{\ell=1}^{\hat{\iota}} \log(\hat{\lambda}_{\ell}) + (n-s-\hat{\iota})\log(\hat{\delta})= \Gamma_{n-s}(W^\top\Diag(\hat y)W).
\end{equation*}
Therefore, the minimum gap between $\hat y$ in \ref{gamma_bound} and feasible solutions of \ref{eq:dual_gamma_dopt_01} of the form $(\hat\Theta,\omega,\nu,\tau)$ is the optimal value of the linear program
\begin{equation}\label{eq:g_theta2}\tag{$\tilde{G}(\hat\Theta)$}
\begin{array}{rrl}
&\min &- \omega^\top a + \nu^\top b + \tau (n-s) - (n-s),\\
&\text{s.t.} 
& \omega - \nu - \tau\mathbf{e} = -\diag(W\hat\Theta W^\top)  ,\\
&&\nu \geq 0, \omega \geq 0.
\end{array}
\end{equation}
To construct an optimal solution $(\tilde{\omega},\tilde{\nu},\tilde{\tau})$ for \ref{eq:g_theta2},  we may use the same   procedure described for \ref{eq:g_theta} in  \S\ref{sec:natural_bound}, 
just replacing $A$ by $W$, $s$ by $n-s$, $l$ by $a$, and $u$ by $b$. 

We note that to obtain the minimum duality gap between $\hat{y}$ in \ref{gamma_bound} and   $(\hat\Theta,\tilde\omega,\tilde\nu,\tilde\tau)$ in \ref{eq:dual_gamma_dopt_01}\,, we should choose $\epsilon$ very close to zero in \eqref{betaepsilon}. In our numerical experiments, considering round-off errors, we set $\epsilon=0$  (see \cite[Sec. 3.4.4.1]{FL2022} for details on a similar analysis on $\epsilon$). We also note that with the dual feasible solution $(\hat\Theta,\tilde\omega,\tilde\nu,\tilde\tau)$ of \ref{eq:dual_gamma_dopt_01}\,, 
the conclusion of Thm. \ref{thm:fix_dopt} still holds,
and we may fix variables in subproblems considered during the execution of the B\&B algorithm for \ref{prob01}, when considering the $\Gamma$-bound. 
Finally, from \cite[Thm. 21]{pfl2024_gmesp_arxiv}, we can conclude that if  $\hat y$  is an optimal solution to \ref{gamma_bound}\,, then $(\hat\Theta,\tilde\omega,\tilde\nu,\tilde\tau)$ is optimal to \ref{eq:dual_gamma_dopt_01}\,.

\subsection{The Data-Fusion problem as a  case of the 0/1 D-Optimality problem }\label{sec:datafusion}

 Next, we exploit the relation between \ref{prob01} and the Data-Fusion problem addressed in \cite{li2022d}.

Let $A := \begin{pmatrix}
    G\\ H
\end{pmatrix}\in\mathbb{R}^{n\times m}$ , with $G \in \mathbb{R}^{p \times m}$, $H \in \mathbb{R}^{q \times m}$. Let $P:=\{1,\ldots,p\}$ and $Q:=N\setminus P$. Assume that $A$ has full column rank and $B:=H^\top H \in \mathbb{S}_{++}^{m}$\,. The  D-optimality problem 
\begin{equation}\label{partdopt}
\max\! \left\{ \ldet \!\left(A^\top \Diag(x) A\right)  :  \mathbf{e}^\top x=s, x_i=1, i\in Q, x\!\in\!\{0,1\}^n\right\},
\end{equation}
where $0<s-q<p$ and  the variables corresponding to the rows of $H$ are  fixed at one, can be reformulated as the following  Data-Fusion problem,
\begin{equation}\label{datafusion}\tag{Fu}
\hypertarget{zfutarget}{\zfuthing}:=\max \left\{ \ldet \left( B  + G^\top \Diag(x) G\right) \, : \, \mathbf{e}^\top x=s-q, \, x\in\{0,1\}^{p}\right\}.
\end{equation}

\subsubsection{Comparison of bounds for the Data-Fusion problem}\label{subsec:datafusionbound}

We consider other  bounds for the Data-Fusion problem \ref{datafusion},
and our  goal in this subsection is to compare them theoretically. 
In similar settings, it is rare to be able to get significant results comparing 
bounds. Exceptionally, we have that the ``factorization bound'' 
dominates the ``spectral bound'' for the 
MESP (see \cite{ChenFampaLee_Fact} and \cite[Thm. 3.4.18]{FL2022}).

To  differentiate the bounds for \ref{prob01} from the bounds for \ref{datafusion}, we denote them respectively by  $\mbox{z}_{\mbox{\protect\tiny $(\cdot)$}}$ and $\mathfrak{z}_{\mbox{\protect\tiny $(\cdot)$}}$\,. 

 \begin{itemize}
 \item The \emph{natural bound} for \ref{datafusion} is defined as 
 \begin{equation}\label{natural}\tag{$\mathcal{N}_{{\mbox{\protect\tiny Fu}}}$}
\begin{array}{llll}
\hypertarget{zntarget}{\znthing}
:= &\max  &\ldet \left( B  + G^\top \Diag(x) G\right),\\
&\text{s.t.} 
&\mathbf{e}^\top x=s-q,\\
&&0\leq x_i\leq 1, \quad i\in P.
\end{array}
\end{equation}

\item The \emph{spectral bound}  
for \ref{datafusion},  from  \cite{KoLeeWayne}, is defined as
    \begin{equation}\label{spectral}\tag{$\mathcal{S}_{{\mbox{\protect\tiny Fu}}}$}
        \hypertarget{zstarget}{\zsthing} := \ldet({B}) +  \textstyle\sum_{i=1}^{s-q} \log \left(1 +  \lambda_i\left(G B^{-1}G^\top\right)\right).
    \end{equation}
\item Let  $L L^\top$ be the Cholesky factorization of $B$. Let $\rho_i(GL^{-\top})$ denote  $\|G_{i\cdot}L^{-\top}\|_2$\,, for $i \in P$, and let $\tau$ be a bijection from $\{1,\ldots,p\}$ to $P$,  such that $\rho_{\tau(i)}(GL^{-\top})$ $ \geq \rho_{\tau(j)}(GL^{-\top})$,  whenever $i \leq j$. The \emph{Hadamard bound} for \ref{datafusion}, from \cite{KoLeeWayne},  is  defined as
\begin{equation}\label{hadamard}\tag{$\mathcal{H}_{{\mbox{\protect\tiny Fu}}}$}
    \hypertarget{zhtarget}{\zhthing} := \ldet({B}) +\textstyle \sum_{i=1}^{s-q} \log\left(1 +  \rho_{\tau(i)}^2\left( G L^{-\top}\right)\right).
\end{equation}

\item 
Let $\Psi \Psi^\top$ be the Cholesky factorization of $I_p + G B^{-1} G^\top$, and $\psi_i^\top$ be the $i$-th row of $\Psi$, for $i\in P:=\{1,\ldots,p\}$. 
The $\Gamma$-\emph{bound}  for  \ref{datafusion}, from \cite{li2022d} (they call it the ``M-DDF bound''), 
is defined as
\begin{equation}\label{zm1a}\tag{$\Gamma_{{\mbox{\protect\tiny Fu}}}$}
\begin{array}{llll}
\hypertarget{zgtarget}{\zgthing}:= \ldet(B) + &\max & \Gamma_{s-q}\left(\textstyle\sum_{i\in P} z_i \psi_i \psi_i^\top\right),\\
&\text{s.t.} 
&\mathbf{e}^\top z = s-q,\\
&&0\leq z_i\leq 1, \quad i\in P.
\end{array}
\end{equation}

\item  
Let $\Phi \Phi^\top$ be the Cholesky factorization of  $I_p-G( B  + G^\top G)^{-1}G^\top$, and  $\phi_i^\top$ be the $i$-th row of $\Phi$, for $i \in P:=\{1,\dots,p\}$. Note that  the greatest eigenvalue of $G( B  + G^\top G)^{-1}G^\top$ 
is strictly less than one, so the matrix $I_p-G( B  + G^\top G)^{-1}G^\top$ is  positive-definite. The 
\emph{complementary $\Gamma$-bound} for \ref{datafusion},
 from \cite{li2022d} (they call it the ``M-DDF-complementary bound''),   is defined as
\begin{equation}\label{zmc1}\tag{comp-$\Gamma_{{\mbox{\protect\tiny Fu}}}$}
\begin{array}{llll}
\hypertarget{zcgtarget}{\zcgthing}:= \ldet(B  + G^\top G) + &\max & \Gamma_{p-(s-q)}\left(\textstyle\sum_{i\in P} z_i \phi_i \phi_i^\top\right),\\
&\text{s.t.} 
&\mathbf{e}^\top z =p - (s - q),\\
&&0\leq z_i\leq 1, \quad i\in P.
\end{array}
\end{equation}

 \end{itemize}

Next we see that \ref{gamma_bound} is a generalization to \ref{prob01} of \ref{zmc1}\,.

\begin{theorem}\label{prop:comp}

The  $\Gamma$-bound  for \eqref{partdopt} is the same as $\zcg$\,.
\end{theorem}

\begin{proof}
We note that  $p - (s-q) =n-s$, and   $ \ldet(B  + G^\top G) = \ldet(A^\top A)$. 
Moreover, we have
\[
\begin{array}{ll}
    \Phi \Phi^\top &= I_p - G(B  + G^\top G)^{-1}G^\top
    ~=~ I_p - A_{P \cdot}(A^\top A)^{-1}A_{P \cdot}^\top\\
    &= I_p - (U\Sigma V^\top)_{P\cdot}(V \Sigma U^\top U\Sigma V^\top)^{-1}((U\Sigma V^\top)_{P\cdot})^\top\\
    &= I_p - U_{P\cdot}\Sigma V^\top((V \Sigma) (\Sigma V^\top))^{-1}(U_{P\cdot}\Sigma V^\top)^\top\\
    &= I_p - U_{P\cdot}(\Sigma V^\top)(\Sigma V^\top)^{-1} (V\Sigma)^{-1}(V\Sigma) U_{P\cdot}^\top\\
    &= I_p - U_{P\cdot}U_{P\cdot}^\top 
    ~=~ (I_n - UU^\top)_{P,P}
     ~=~ W_{ P\cdot} W_{ P\cdot}^\top\,,
\end{array}
\]
recalling that  $W$ was defined for \ref{gamma_bound}\,, $\Phi$ was defined for \ref{zmc1}\,, and $A=U\Sigma V^\top$ is the real singular-value decomposition of $A$.

Let $\lambda^+(X)$ represent the set of positive eigenvalues of $X$, then
\[
\begin{array}{ll}
&\lambda^+(\textstyle\sum_{i\in P} z_i \phi_i \phi_i^\top ) ~=~ \lambda^+(\Phi^\top \Diag(z)\Phi)
     ~=~ \lambda^+(\Phi^\top \Diag(z)^{1/2}\Diag(z)^{1/2}\Phi)\\
    &\quad = \lambda^+(\Diag(z)^{1/2}\Phi \Phi^\top \Diag(z)^{1/2})
    ~=~ \lambda^+(\Diag(z)^{1/2} W_{P\cdot } W_{P\cdot }^\top\Diag(z)^{1/2})\\
    &\quad = \lambda^+(W_{P\cdot }^\top \Diag(z)^{1/2}\Diag(z)^{1/2}W_{P\cdot})
    ~=~ \lambda^+(W_{P\cdot }^\top \Diag(z)W_{P\cdot })\\
    &\quad = \lambda^+(\textstyle\sum_{i\in P} z_i w_i w_i^\top).
\end{array}
\]
Therefore,  $\Gamma_{n-s}(\sum_{i\in P} z_i \phi_i \phi_i^\top )= \Gamma_{n-s}(\sum_{i\in P} z_i w_i w_i^\top)$, and \ref{zmc1} is equivalent to
\begin{equation*} 
\begin{array}{llll}
 \ldet(A^\top A) + &\max & \Gamma_{n-s}(\textstyle\sum_{i\in P} z_i w_i w_i^\top),\\
&\text{s.t.} 
&\mathbf{e}^\top z = n-s,\\
&&0\leq z_i\leq 1, \quad i\in P.
\end{array}
\end{equation*} 
Finally, as $x_i=1$, for $i\in Q$ in \ref{prob01}, we have $y_i=0$, for $i\in Q$ in its corresponding \ref{gamma_bound}\,, which becomes
\begin{equation*} 
\begin{array}{lllll}
&\textstyle 2 \sum_{i= 1}^m \log(\Sigma_{ii}) + &\max  &\Gamma_{n-s}(\sum_{i\in P} y_i w_i w_i^\top),\\
&&\text{s.t.} 
&\mathbf{e}^\top y = n-s,\\
&&&0\leq y_i\leq 1, \quad i\in P.
\end{array}
\end{equation*}
As $2\sum_{i= 1}^m \log(\Sigma_{ii}) = \ldet(A^\top A)$,  we verify that the bounds  are equivalent. \qed
\end{proof}

Considering now the bounds presented for \ref{datafusion}, we note that in general we have no dominance among some of bounds, as  illustrated in what follows.

\begin{example}\label{ex:zn_zs_zh}
 Let  
\[
G:= \begin{pmatrix*}[r]
1~&~-1~&\quad 1\\
1~&~0~&\quad 1\\
-1~&~0~&\quad 1\\
1~&~1~&\quad 1\\
1~&~0~&\quad 0
\end{pmatrix*}
~ \mbox{ and } ~ 
H:= \begin{pmatrix*}[r]
0~&~1~&~0\\
-1~&~1~&~-1\\
1~&~-1~&~0\\
\end{pmatrix*}.\]
In Table \ref{tab:example_bounds_a}, we present the bounds $\zn$\,,  $\zs$ and  $\zh$ for \ref{datafusion}, with the data defined by $G$, $B:=H^\top H$, and $s-q\in\{1,2,3\}$.
\begin{table}[!ht]
\centering
\begin{tabular}{c|ccc}
$s-q$ & $\zn$ & $\zs$ & $\zh$ \\ \hline
1     & 2.622     & 2.324           & 1.946           \\ \hline
2     & 3.714      & 4.302           & 3.738           \\ \hline
3     & 4.205     & 4.745           & 4.836           
\end{tabular}
\caption{Comparison of bounds for \ref{datafusion}}
\label{tab:example_bounds_a} 
\end{table}

We note that for $s-q = 1$, we have $\zh <  \zs < \zn$\,, for $s-q = 2$, we have
$\zn < \zh< \zs$\,, and  for $s-q = 3$, we have $\zn  < \zs < \zh$\,, confirming the nondominance among these three bounds. 
\end{example}

\begin{example} 
Let  
\[
G:= \begin{pmatrix*}[r]
1~&~0~&\quad 1\\
0~&~-1~&\quad 0\\
1~&~1~&\quad 0\\
0~&~1~&\quad 1\\
-1~&~-1~&\quad -1\\
\end{pmatrix*}
~ \mbox{ and } ~ 
H:= \begin{pmatrix*}[r]
0~&~1~&~0\\
-1~&~1~&~-1\\
1~&~-1~&~0\\
\end{pmatrix*}.\]
In Table \ref{tab:example_bounds_b}, we present the bounds $\zn$\,, $\zcg$\,, $\zh$ and  $\zg$ for \ref{datafusion} with the data defined by $G$, $B:=H^\top H$, and $s-q\in\{1,2\}$.
\begin{table}[!ht]
\centering
\begin{tabular}{c|cccc}
$s-q$ & $\zn$ & $\zcg$  & $\zh$ & $\zg$\\ \hline
1     & 2.174 & 2.024                & 1.792  &  1.792        \\ \hline
2     & 3.162 & 3.174                & 3.584  &  3.196        
\end{tabular}
\caption{Comparison of bounds for \ref{datafusion}}
\label{tab:example_bounds_b}
\end{table}

We note that
for $s-q = 1$, we have  $ \zcg< \zn$ and $ \zcg> \zh$\,. On the other side,  for $s-q = 2$, we have
 $ \zcg > \zn$ and $ \zcg< \zh$\,, confirming  nondominance between $ \zcg$ and neither $  \zn$ nor $  \zh$\,.
 
Moreover, for $s-q = 1$, we have  $ \zg< \zn$ and $ \zg< \zcg$\,. On the other side,  for $s-q = 2$, we have
 $ \zg > \zn$ and $ \zg> \zcg$\,,  confirming the nondominance between $ \zg$ and neither $  \zn$ nor $  \zcg$\,.
\end{example}

Despite the nondominance shown in these examples, we summarize here some relations  that we will establish in the remainder of this section.
\begin{itemize}
\item For $s-q\geq 1$, we have $\zg \leq \zs$\,.
\item For $s-q\geq 1$, we have $\zcg \leq \zs$\,.
    \item For $s-q \geq m\geq 1$, we have $\zn \leq \zs$\,, but for $s-q = m-1$, $m\geq 2$, we can have $\zn> \zs$\,.
    \item For  $s -q = 1$, we have $\zh = \zfu$\,, where $\zfu$ is the optimal value  of \ref{datafusion}.   
    \item For  $s-q\geq 1$, we have $\zs - \zh \leq  \textstyle\sum_{i =s-q +1}^{p} ~\log \left(\frac{d_i\left(I_p + G B^{-1}G^\top\right)}{\lambda_i\left(I_p + G B^{-1}G^\top\right)}\right)$.
    \item For $s-q\geq m\geq 1$, we have $\zs - \zh \leq  \textstyle\sum_{i =s-q +1}^{p} ~\log \left(d_i\left(I_p + G B^{-1}G^\top\right)\right)$.
    \item For  $s-q\geq 1$, we have   
$
\zg - \zh \leq \textstyle\sum_{i =2}^{s-q} \log \left(\frac{d_{1}\left(I_p + G B^{-1}G^\top\right)}{d_i\left(I_p + G B^{-1}G^\top\right)}\right).
$
\end{itemize}

\begin{lemma}\label{spec_ldetAAa}
For $s-q \geq 1$, we have   
\[
\ldet(B+G^\top G) = \zs + \textstyle\sum_{i=s-q+1}^p ~\log(1+\lambda_i(GB^{-1}G^\top )).
\]
\end{lemma}
\begin{proof}

\begin{align*}
     \ldet(B + G^\top G) 
     &= \ldet(B) + \ldet(I_m + B^{-1/2}G^\top G B^{-1/2})\\
     &= \ldet(B) + \ldet(I_p + G B^{-1} G^\top)\\
     &= \ldet(B) + \textstyle\sum_{i = 1}^p \log(\lambda_i(I_p + G B^{-1} G^\top))\\
     &= \zs  + \textstyle\sum_{i = s-q+1}^p \log(1+\lambda_i(G B^{-1} G^\top)),
\end{align*}
where the second equation follows from the fact that the matrices $B^{-\frac{1}{2}}G^\top GB^{-\frac{1}{2}}$ and $GB^{-1}G^\top$ have the same positive eigenvalues.
 \qed
\end{proof}

\begin{theorem}\label{spec_ldetAA}
For $s-q \geq m\geq 1$, we have   $\zs = \ldet(B+G^\top G)$.
\end{theorem}
\begin{proof}
As  $\rank(GB^{-1}G^\top)= \rank(B^{-\frac{1}{2}}G^\top GB^{-\frac{1}{2}})=m$, we have 
\[
\textstyle\sum_{i = s-q+1}^p \log(1+\lambda_i(G B^{-1} G^\top))=\sum_{i = s-q+1}^p \log(1)=0.
\]
The result then follows from Lem. \ref{spec_ldetAAa}.
 \qed
\end{proof}

Motivated  specifically by the
dominance of the ``spectral bound''
by the  ``factorization bound'' 
for the 
Maximum-Entropy Sampling Problem (see \cite{ChenFampaLee_Fact} and \cite[Thm. 3.4.18]{FL2022}), we 
have the following result. 

\begin{theorem}\label{dominance_zm_speca}
    For  $s-q\geq 1$, we have $\zg \leq \zs$\,.
\end{theorem}
\begin{proof}

The Lagrangian dual  of 
	 \ref{zm1a} is 
\begin{equation}\label{dualzm1a}
\begin{array}{lll}
\ldet {B} +&\min &-\textstyle\sum_{i = p - (s-q) + 1}^{p} \log(\lambda_{i} (\Theta ))    + \textstyle\sum_{i=1}^{p}\nu_i + \tau (s-q) - (s-q),\\
&\text{s.t.} 
&\diag(\Psi \Theta \Psi^\top)  + \omega - \nu - \tau\mathbf{e} = 0,\\
&&\Theta \succ 0,\, \nu \geq 0,\, \omega \geq 0.
\end{array}
\end{equation} 
Then, it suffices to construct a feasible solution to \eqref{dualzm1a} with objective value  equal to $\zs$\,. 
We consider $\hat\Theta := (\Psi^\top \Psi)^{-1}$, 
$\hat\omega := 0$, $\hat\nu := 0$ and $\hat\tau := 1$, and
we can verify that $(\hat\Theta,\hat\nu,\hat\omega,\hat\tau)$ is a feasible solution to \eqref{dualzm1a} with objective value  
\begin{align*}
   & \ldet(B) - \textstyle\sum_{i =p-(s-q) +1}^{p} \log(\lambda_{i} (\hat\Theta ))  ~=~  \ldet({B}) +  \textstyle\sum_{i=1}^{s-q} \log (\lambda_{i} (\Psi^\top \Psi )) \\
    &\quad = \ldet({B}) +  \textstyle\sum_{i=1}^{s-q} \log (\lambda_{i} (\Psi \Psi^\top ))
    ~=~\ldet({B}) +  \textstyle\sum_{i=1}^{s-q} \log (  \lambda_i(I_p + G B^{-1} G^\top))\\
    &\quad = \zs\,.   
\end{align*}

\vspace{-15pt}
\qed
\end{proof}

\begin{theorem}\label{dominance_zmc_speca}
For $s-q\geq 1$, we have $\zcg \leq \zs$\,.
\end{theorem}
\begin{proof}
    The Lagrangian dual  of 
	 \ref{zmc1} is 
\begin{equation}\label{dualzmc}
\begin{array}{lll}
\ldet(B \! + \!G^\top G ) +&\!\min &- \!\!\!\!\!\!\!\displaystyle\sum\limits_{i =s-q +1}^{p} \!\!\!\!\!\log(\lambda_{i} (\Theta ))   \! +\! \displaystyle\sum_{i\in P}\nu_i + \tau (p\!+\!q\!-\!s) \!-\! (p\!+\!q\!-\!s),\\[7pt]
&\text{s.t.} 
&\diag(\Phi \Theta \Phi^\top)  + \omega - \nu - \tau\mathbf{e} = 0,\\[5pt]
&&\Theta \succ 0,\nu \geq 0, \omega \geq 0.
\end{array}
\end{equation}
It suffices to construct a feasible solution to \eqref{dualzmc} with objective value  equal to $\zs$\,. From the Woodbury-Formula, we have
\begin{align*}
    (I_p + GB^{-1}G^\top)^{-1} &= I_p - G(B + G^\top G)^{-1}G^\top,
\end{align*}
where the matrix on the right-hand side is equal to $\Phi\Phi^\top$. 

Then, we consider $\hat\Theta := (\Phi^\top \Phi)^{-1}$, 
$\hat\omega := 0$, $\hat\nu := 0$ and $\hat\tau := 1$, and we can  verify that $(\hat\Theta,\hat\nu,\hat\omega,\hat\tau)$ is a feasible solution to \eqref{dualzmc} with objective value
\[
\begin{array}{l}
    \ldet(B + G^\top G ) - \textstyle\sum_{i =s-q +1}^{p} \log(\lambda_{i} (\hat\Theta ))   \\
     \, =\zs + \textstyle\sum_{i=s-q+1}^p \log(1+\lambda_i(GB^{-1}G^\top ))  - \textstyle\sum_{i =s-q +1}^{p} \log(\lambda_{i} (\Phi^{-1}\Phi^{-\top} )) \\
     \, =\zs + \textstyle\sum_{i=s-q+1}^p \log(1+\lambda_i(GB^{-1}G^\top ))  - \textstyle\sum_{i =s-q +1}^{p} \log(\lambda_{i} (\Phi^{-\top}\Phi^{-1} )) \\
    \, =\zs + \textstyle\sum_{i=s-q+1}^p \textstyle\log(1+\lambda_i(GB^{-1}G^\top )) - \textstyle\sum_{i =s-q +1}^{p} \log(\lambda_{i} (I_p + G B^{-1} G^\top)) =\zs\,,
    \end{array}
    \]
where the first equation follows from Lem. \ref{spec_ldetAAa}.
\qed
\end{proof}

\begin{theorem}\label{prop:compnatural}
   For $s-q \geq m\geq 1$, we have $\zn \leq \zs$\,. 
\end{theorem}
\begin{proof}
By considering Thm. \ref{spec_ldetAA}, it suffices to show that there is no feasible solution to \ref{natural} with objective value greater than  $\ldet(A^\top A)$. Let $\hat{x}$ be a  feasible solution  to \ref{natural}\,, then
\[
\det(A^\top \diag(\hat x) A  )
\leq \det(A^\top \diag(\hat x) A  + A^\top \diag(\mathbf{e}-\hat x) A  )=\det(A^\top A).
\]
For the inequality above, see \cite[Lem. 1.2]{SchurBook}. \qed 
\end{proof}

It is natural to wonder whether the condition $s-q\geq m$ in Thm. \ref{prop:compnatural} is necessary. Next, we demonstrate that the conclusion of Thm. \ref{prop:compnatural}  does not generally hold when $s-q= m-1$. 

\begin{theorem}
    For $s-q = m-1$, $m\geq 2$, we can have $\zn> \zs$\,.
\end{theorem}
\begin{proof}
Let $p=m$, $B=G=I_m$, $s-q=m-1$. In this case, we have 
\[
\zs = (m-1)\log(2) \qquad \mbox{ and }\qquad 
\zn =
\left(2 - 1/m\right)^m,
\]
the latter of which holds because the symmetry of  this instance implies that \ref{natural} has the symmetric solution $x^*_i=1-1/m$ for all $i\in P$.

We want to show that $\zn> \zs$\,, or
$
f(m):=\left(2 - 1/m\right)^m/(m-1) > \log(2)
$.
It is easy to check that $f(3)>f(2)>\log(2)$, so it suffices to show that the $f$ is increasing in $m$ for $m\geq 3$. This is equivalent to prove that  the $\log(f)$ is increasing in $m$. We have
\[
\frac{d(\log(f))}{dm}=\frac{-m+\left(2m^2-3m+1\right)\log\left(2-1/m\right)}{(m-1)(2m-1)}.
\]
The denominator above is clearly positive and the numerator is bounded below by the convex quadratic
$
-m+\left(2m^2-3m+1\right)/2,
$
which has roots $\frac{1}{4}(5\pm\sqrt{17})$. Both roots are smaller than $3$, and therefore $f$ is increasing for $m\geq 3$. \qed
\end{proof}

\begin{theorem}\label{hada_optimal_s1}
 For  $s -q = 1$, we have $\zh = \zfu$, where $\zfu$ is the optimal solution value  of \ref{datafusion}.   
\end{theorem}
\begin{proof}
Let $\Upsilon_{i\cdot} := G_{i\cdot}L^{-\top}$, for $i \in P$. For $s -q= 1$, we have
\[\zh =  \ldet(B) +\textstyle\max_{ i\in P} \log\left(1 +  \|\Upsilon_{i\cdot}\|_2^2\right),\]
and
\[\zfu =  \textstyle\max_{ i\in P} \ldet \left(B + G^\top_{i\cdot} G_{i\cdot} \right).\]
We have
\begin{align*}
   &B +  G^\top_{i\cdot} G_{i\cdot} ~=~ L L^\top+   G^\top_{i\cdot} G_{i\cdot}
    ~=~ L L^\top + \left(L L^{-1}\right) G^\top_{i\cdot} G_{i\cdot}\left(L^{-\top} L^\top\right)\\
    &\quad= L L^\top + L \Upsilon_{i\cdot}^\top  \Upsilon_{i\cdot} L^\top
    ~=~ L \left(I + \Upsilon^\top_{i\cdot} \Upsilon_{i\cdot}\right)L^\top,
\end{align*}
so 
\begin{align*}
    &\zfu 
    ~=~\textstyle\max_{ i\in P} \ldet \left(  L \left(I + \Upsilon_{i\cdot}^\top \Upsilon_{i\cdot}\right)L^\top \right)
    ~=~ \textstyle\max_{ i\in P} \ldet \left(\left(  L L^\top\right)\cdot \left(I + \Upsilon_{i\cdot}^\top \Upsilon_{i\cdot}\right)\right) \\
     &\quad =~ \ldet(B) + \textstyle\max_{ i\in P}\ldet\left(I + \Upsilon_{i\cdot}^\top \Upsilon_{i\cdot}\right)
    ~=~ \ldet(B) + \textstyle\max_{ i\in P}\log\left(1 + \Upsilon_{i\cdot} \Upsilon_{i\cdot}^\top\right)\\
    &\quad =~ \ldet(B) + \textstyle\max_{ i\in P}\log\left(1 + \|\Upsilon_{i\cdot} \|_2^2\right) ~=~ \zh\,.
\end{align*}

\vspace{-17pt}
\qed
\end{proof}

\begin{lemma}\label{lem:diagelements} For $s-q\geq 1$, we have
    \[    \zh = \ldet({B}) + \sum_{i=1}^{{s-q}} \log\left( d_i\left( I_p +  G B^{-1}G^\top\right)\right).
    \]
    
\end{lemma}
\begin{proof}
We have
    \begin{align*}
        \zh &= \ldet(B) + \textstyle\sum_{i=1}^{{s-q}} \log\left(1 +  \rho_{\tau(i)}^2\left( G L^{-\top}\right)\right)\\
        &= \ldet(B) + \textstyle\sum_{i=1}^{{s-q}} \log\left(1 + d_i\left( (G L^{-\top})(G L^{-\top})^{\top}\right)\right)\\
        &= \ldet({B}) + \textstyle\sum_{i=1}^{{s-q}} \log\left( d_i\left( I_p + G B^{-1}G^\top\right)\right).  
    \end{align*}

    \vspace{-17pt}
    \qed
\end{proof}

\begin{theorem}\label{lem:general_inequality_spec_hada}
For  $s-q\geq 1$, we have $\zs - \zh \leq  \!\!\!\!\displaystyle\sum\limits_{i =s-q +1}^{p} \!\!\!\!\!\log \left(\frac{d_i\left(I_p + G B^{-1}G^\top\right)}{\lambda_i\left(I_p + G B^{-1}G^\top\right)}\right)$.
\end{theorem}
\begin{proof}
    From Hadamard's inequality (see \cite[Thm. 7.8.6]{HJBook} for the more general Oppenheim's inequality), we have
    \[
    \begin{array}{ll}
        &\textstyle \prod_{i\in P} \lambda_i\!\left(I_p \!+\! G B^{-1}G^\top\right) \leq \prod_{i\in P} d_i\!\left(I_p \!+\! G B^{-1}G^\top\right)\Rightarrow\\[3pt]
        &\textstyle\sum_{i\in P}\log \left(\lambda_i\!\left(I_p \!+\! G B^{-1}G^\top\right)\!\right) \leq \sum_{i\in P}\log \left(d_i\!\left(I_p \!+\! G B^{-1}G^\top\right)\!\right)\Rightarrow\\[3pt]
        &\textstyle\ldet(B) \!+\!\sum_{i\in P}\log \!\left(\lambda_i\!\left(I_p \!+\! G B^{-1}G^\top\right)\!\right) \!\leq \ldet(B) \!+\!\! \sum_{i\in P}\log\! \left(d_i\!\left(I_p \!+\! G B^{-1}G^\top\right)\!\right)\Rightarrow\\[3pt]
        &\textstyle\zs + \!\sum_{i =s-q +1}^{p}  \!\log \left(\lambda_i\!\left(I_p \!+\! G B^{-1}G^\top\right)\!\right) \leq \zh + \sum_{i =s-q +1}^{p} \log \left(d_i\!\left(I_p \!+\! G B^{-1}G^\top\right)\!\right).
    \end{array}
    \]
Then, the result follows.  \qed
\end{proof}

\begin{corollary}\label{cor:valid_inequality_hadamard}
   For $s-q\geq m\geq 1$, we have 
   \[
   \zs - \zh \leq \textstyle\sum_{i =s-q +1}^{p} \log \left(d_i\!\left(I_p \!+\! G B^{-1}G^\top\right)\right).
   \]
\end{corollary} 
\begin{proof}
    As $\rank(GB^{-1}G^\top) = m$, we have $\lambda_i(I_p \!+\! GB^{-1}G^\top) = 1$ for $i = \{m+1,\dots,p\}$. 
    Then, the result 
    follows  from Thm. \ref{lem:general_inequality_spec_hada}.  \qed
\end{proof}

\begin{theorem}\label{thm:difgamahadamard}
For  $s-q\geq 1$, we have   
\[
\zg - \zh \leq \!\displaystyle\sum\limits_{i =2}^{s-q} \log \left(\frac{d_{1}\left(I_p + G B^{-1}G^\top\right)}{d_i\left(I_p + G B^{-1}G^\top\right)}\right).
\]
\end{theorem}

\begin{proof}
From Lem. \ref{lem:diagelements}, we see that it 
suffices to construct a feasible solution to \eqref{dualzm1a} with objective value  equal to $\ldet(B) + \sum_{i = 1}^{s-q}\log(d_{1}\left(I_p + G B^{-1}G^\top\right))$\,. 
We let $\hat{d}_1 := d_{1}\left(I_p + G B^{-1}G^\top\right)$, and we consider  $\hat\Theta := \frac{1}{\hat{d}_1}I$, $  \hat\omega :=   \mathbf{e} - \frac{1}{\hat{d}_1}\diag(\Psi \Psi^\top)$, $\hat\nu := 0$ and $\hat\tau := 1$. Recalling that  $\Psi \Psi^\top$ is the Cholesky factorization of $I_p + G B^{-1} G^\top$ (see \ref{zm1a}),  we see that $\hat\omega\geq 0$. Thus, we can verify that $(\hat\Theta,\hat\omega,\hat\nu,\hat\tau)$ is a feasible solution to \eqref{dualzm1a} with objective value  
\vspace{-15pt}

\begin{align*}
   & \ldet(B) - \!\!\!\!\!\sum_{i =p-(s-q) +1}^{p}\!\!\!\!\! \log(\lambda_{i} (\hat\Theta ))  ~=~  \ldet({B}) +  \sum_{i=1}^{s-q} \log \left(d_{1}\left(I_p + G B^{-1}G^\top\right)\right).  
\end{align*}

\vspace{-20pt}
\qed
\end{proof}

\medskip

 We show in Thm. \ref{hada_optimal_s1} that when $s-q = 1$, we have $\zh = \zfu$. 
 In \cite{li2022d} they show that when $s-q = 1$, we also have $\zg= \zfu$.  
 In Thm. \ref{thm:difgamahadamard} we establish a bound on the difference between $\zg$ and $\zh$. In our numerical experiments,  for $s = 1$, we have $\zg = \zh$\,,  and for $s > 1$, we  always have $\zg < \zh$\,. So, we also have the following conjecture. 

\begin{jconjecture}
For  $s-q\geq 1$, we have  $\zg \leq \zh$\,. 
\end{jconjecture}

\section{Local-search heuristics}\label{sec:heur}
We describe heuristics  to construct a feasible solution to \ref{prob} by applying  a local-search procedure from an initial solution. The basic idea, known as the Fedorov Exchange Algorithm, is well known especially for \ref{prob01}, and 
is contained in, for example \cite{Miller} and \cite{Atkinson}, and the many references therein. 
For completeness, and reproducibility considerations, we give a lot of details that are not all new.
We describe ours ways of constructing the initial solution (\S4.1-2.) and  performing the local search (\S4.3); see \cite{HARMAN2020},
for alternative ideas (for the case of ``saturated designs'', i.e., $s=m$). 
We also briefly explain how we update the objective value  inside the local search in order to make it more efficient (\S4.4).  
Without loss of generality, we assume, in this section, that $l=0$ in \ref{prob}. 

\subsection{Initial solutions from the SVD decomposition of $A$}\label{subsec:heur1}

 Next, we  show how we obtain  initial solutions for our local-search  procedures from the full real singular-value decomposition (SVD) $A=U\Sigma V^\top$ (see \cite{GVL1996}, for example), where $U\in\mathbb{R}^{n\times n}$, $V\in\mathbb{R}^{m\times m}$ are orthonormal matrices and $\Sigma=\mathrm{diag}(\sigma_1,\sigma_2,\dots,\sigma_m)\in\mathbb{R}^{n\times m}$ with  singular values $\sigma_1\ge\sigma_2\ge\dots\ge\sigma_m\ge0$.
 
  First, to ensure that we start the local-search procedures with a feasible solution for \ref{prob} with finite objective value, we construct a vector $\tilde{x}\!\in\!\{0,1\}^n$, such that $\mathbf{e}^\top \tilde{x}=m$ and $A^\top\mbox{diag}(\tilde{x})A\!\succ\! \!0$.   
   This is equivalent to choosing $m$ linearly independent rows of $A$, and setting $\tilde{x}$ as the incidence vector for the selected subset of rows.  To select the  rows,  we use the Matlab function nsub\footnote{\url{www.mathworks.com/matlabcentral/fileexchange/83638-linear-independent-rows-and-columns-generator}}  (see \cite{FLPX2021} for details). We denote the set of indices of the selected rows by $\tilde{N}$.

 We note that 
 for each  $k \in N$, we have $\sum_{j\in N} U_{jk}^2 = 1$ and $\sum_{j\in N} U_{kj}^2 = 1$. Then, defining
\[
\textstyle
x^0_j := \sum_{k=1}^s U_{jk}^2\,,\quad   j\in N,
\]
we have  $\mathbf{e}^\top x^0 = s$ and  $0 \leq x^0 \leq \mathbf{e}$. So, $x^0$ is a feasible solution of \ref{cont_rel}\,.  

We let $\tau$ be the permutation of the indices in $N$, such that      $x^0_{\tau(1)}\geq x^0_{\tau(2)}\geq\cdots\geq x^0_{\tau(n)}$\,.
Then, we propose two procedures to construct a feasible solution $\bar{x}$ for \ref{prob}, considering  $\tau$ and $\tilde{x}$.
\begin{itemize}
    \item `Bin$(x^0)$':   
    Let $\bar{N}$ be the first $s-m$ indices in $\tau$ (which depends on $x^0$) that are not in $\tilde {N}$. Set  $\bar{x}_{j}:=1$, for $j\in\bar{N}$, and $\bar{x}_{j}:=\tilde{x}_j$\,, for $j\notin \bar{N}$.
     \item `Int$(x^0)$': 
     Let $\Delta =u-\tilde{x}$ and $\bar{s}=s-m$. Define, for all $j\in N$, 
     \[
     \textstyle
     \tilde y_{\tau(j)} :=  \min\left\{\Delta_{\tau(j)}\,,\,\max\left\{0,\bar{s}-\sum_{i=1}^{j-1}\tilde{y}_{\tau(i)}\right\}\right\}.
     \]
     Then, set $\bar{x}:=\tilde{x}+\tilde{y}$.
\end{itemize}

\begin{example}\phantom{.}\break 
\vspace{-10pt}

\noindent Bin$(x^0)$: 
 Suppose that $m=3$ and $s=5$, and we need to complete 
 \[
\tilde{x}:=(0,1,0,0,1,1,0)
 \]
 to a binary vector, with $s=5$ ones.
We do this by choosing $s-m=2$ more coordinates,
in the order of magnitude in 
\[
x^0:=(1,.7,.5,.3,.2,0,0),
\]
resulting in 
\[
\bar{x}:=(1,1,1,0,1,1,0).
\]

\noindent Int$(x^0)$: Suppose now that we have the same $x^0$, but 
$s=10$ (so $s-m=7$) and all upper bounds equal to 3.
Then 
\[
\tilde{y}:=(3,2,2,0,0,0,0),
\]
and
\[
\bar{x}:=(3,3,2,0,1,1,0).  \qed
\]
\end{example}

Next, we note that the objective of \ref{prob} is equal to  $\ldet (\Sigma^{\top}U^{\top}\Diag(x)U\Sigma)$, and so  the choice of $x$ is related to the rows of $U\Sigma$. So, we  also define 
    \[
    \textstyle
    {\hat{x}^0_j := \sum_{i=1}^{m} \left(U_{ji} \Sigma_{ii}\right)^2,\quad   j\in N.}
    \]
    Finally, replacing $x^0$ by $\hat{x}^0$ on the procedures described above, we construct two alternative initial solutions to our local-search procedures. We note that although $\hat{x}^0\geq 0$, it need not be  feasible for \ref{cont_rel}\,.

\subsection{Initial solution  from the continuous relaxation}
We also compute an initial solution to our local-search procedures from a non integer solution $\hat{x}$ to a  relaxation of \ref{prob}, using the following rounding procedure.
\begin{itemize}
    \item `Round$(\hat x)$':   
    Let $\bar{x} := \lfloor \hat{x} \rfloor$, $x^f:=\hat{x}-\bar{x}$, and $\tau$ be the permutation of the indices in $N$, such that      $x^f_{\tau(1)}\geq x^f_{\tau(2)}\geq\cdots\geq x^f_{\tau(n)}$\,.
    Let $\bar{N}$ be the first $s-\mathbf{e}^\top \bar{x}$ indices in $\tau$. Set  $\bar{x}_{j}:=\bar{x}_{j} + 1$, for $j\in\bar{N}$.
\end{itemize}

\subsection{Local-search procedures}\label{subsec:heur2}

Next, we present the three local-search procedures  that we experimented with. They consider as the criterion for improvement of the given solution, the increase in the  value of the objective function of \ref{prob}. The neighborhood of a given solution $\bar{x}$ is defined by
\[
\begin{array}{l}
\mathcal{N}(\bar{x}):= \{y\in\mathbb{Z}^n:
 0\leq y\leq u,~y_i=\bar{x}_i+1,~y_j=\bar{x}_j-1,~ y_k=\bar{x}_k\,,\hphantom{vvvv}\\
\multicolumn{1}{r}{ \forall k\neq i,k\neq j, i\neq j,  i,j,k\in N\},}
\end{array}
\]
and the procedures are described in the following. 
\begin{itemize}
    \item `FI' (Local Search First Improvement): 
    Starting from $\bar{x}$, the procedure visits the solution in $\mathcal{N}(\bar{x})$ with increased objective value with respect to $\bar{x}$, such that $i$ is the least possible index, and $j$ is the least possible index for the given $i$.   
    \item `FI$^+$' (Local Search First Improvement Plus): 
       Starting from $\bar{x}$, the procedure visits the solution in $\mathcal{N}(\bar{x})$ with increased objective value with respect to $\bar{x}$, such that $i$ is the least possible index, and $j$ is selected in $N$, as the  index that maximizes the objective value, for the given $i$.   
    \item `BI' (Local Search Best Improvement): 
       Starting from $\bar{x}$, the procedure visits the solution in $\mathcal{N}(\bar{x})$ with increased objective value with respect to $\bar{x}$, such that $i$ and $j$ are selected in $N$, as the pair of indices that maximizes the objective value.  
\end{itemize}
 For more details on the local-search algorithms, see \cite[Alg. 4]{PonteFampaLeeSBPO22}.

\subsection{Fast local search}\label{sec:fast}

An efficient local search for \ref{prob} is
based on fast computation of $\ldet(B+v_iv_i^\top -v_jv_j^\top)$,
already knowing  $\ldet B$, where $B:=
\sum_{\ell\in N}  \bar{x}_\ell v_\ell v_\ell^\top$,
for some $\bar x$ that is feasible for \ref{prob} such that 
$\bar{x}+\mathbf{e}_i-\mathbf{e}_j$  is
also feasible. 
If $\ldet\left(B+v_iv_i^\top -v_jv_j^\top\right)$ $ > \ldet B$,
then $\bar{x}\!+\!\mathbf{e}_i\!-\!\mathbf{e}_j$ is an improvement on 
$\bar{x}$ in \ref{prob}. 

The well-known Sherman-Morrison formula
\begin{equation}\label{invSM}
(M+ab^\top)^{-1} = M^{-1} -\frac{(M^{-1}a)(b^\top M^{-1})}{(1+b^\top M^{-1}a) }
\end{equation}
and the well-known  matrix-determinant lemma
\begin{equation}\label{detSM}
\det(M+ab^\top) = (1+b^\top M^{-1}a) \det(M)
\end{equation}
are useful for rank-one updates of inverses and determinants, respectively,
in $\mathcal{O}(m^2)$ for an order-$m$ matrix. Use of these equations for
fast updates in local-search algorithms can be found, for example, in \cite{Miller}, \cite{Bohning}, and \cite[Chap. 12]{Atkinson}.

We apply these formulas in the implementation of our local-search procedures. Outside the inner loop of the procedures, that is, the loop in the index $j$,  
we calculate the inverse of $B +v_iv_i^\top$  from the inverse of $B$, using the Sherman-Morrison formula
(setting $M:=B$, $a:=b:=v_i$).
Inside the inner loop (with $i$ fixed), for each $j$ we calculate $\ldet( B)$ 
from the inverse of   $B +v_iv_i^\top$, using the matrix-determinant lemma
(setting $M:= B +v_iv_i^\top$\,,~ 
$a:=-v_j$ and $b:=v_j$).

\section{Other local-search directions}\label{sec:LSdirections}

Next, we investigate more general directions and step sizes to be used in our local-search procedures. These ideas cannot be found in the literature of exchange algorithms. 
We defer proofs to the Appendix.
Given a feasible solution  $\bar x$ for \ref{prob} and an integer  direction $d\in \mbox{n.s.}(\mathbf{e}^\top)\cap \mathbb{Z}^n\setminus\{0\}$
with relatively-prime components, let    $B:=
\sum_{\ell\in N}  \bar{x}_\ell v_\ell v_\ell^\top \in \mathbb{S}_{++}^m$ and $V:=\sum_{\ell\in N}  d_\ell v_\ell v_\ell^\top\in \mathbb{S}^m$. We are interested in the  integer line-search problem; that is,
choosing a step size $k$, 
so that moving from $\bar x$ in the direction $d$ using the  step size $k$, 
the new point 
$\bar{x}+k d$  is
also feasible for \ref{prob}, and $g(k):=\ldet(B + kV)$ is maximized. 
For feasibility,  $k$ must be in the interval
$[k_{\min},k_{\max}]$, where
\[
\begin{array}{l}
 k_{\min}:=\left\lceil\max\left\{\displaystyle\max_{\ell:d_\ell>0}\left\{\textstyle\frac{l_\ell-\bar{x}_\ell}{d_\ell}\right\},\max_{\ell:d_\ell<0}\left\{\textstyle\frac{u_\ell-\bar{x}_\ell}{-d_\ell}\right\}\right\}\right\rceil,\\[10pt]
k_{\max}:=\left\lfloor\min\left\{\displaystyle\min_{\ell:d_\ell>0}\left\{\textstyle\frac{u_\ell-\bar{x}_\ell}{d_\ell}\right\},\min_{\ell:d_\ell<0}\left\{\textstyle\frac{l_\ell-\bar{x}_\ell}{-d_\ell}\right\}\right\}\right\rfloor.
\end{array}
\]
The function $g$ is strictly concave, therefore the optimal $k$,
for a given $V$, will be at one of the endpoints of the interval $[k_{\min},k_{\max}]$,
or at the round up or round down of the stationary point of $g$,
if either is in the interval. 

In the interest of
efficiency for carrying out the line search, we investigate cases 
where for the given $V$, it is possible to compute the stationary point of $g$  with a closed formula.
We have that  
\[
    g(k) = \ldet(B^{\frac{1}{2}}(I+kB^{-\frac{1}{2}}VB^{-\frac{1}{2}})B^{\frac{1}{2}}) = \ldet(B) + \ldet(I+kB^{-\frac{1}{2}}VB^{-\frac{1}{2}}).
\]
We denote by $\lambda_1\!\geq\! \lambda_2\!\geq\! \cdots\!\geq\! \lambda_r$, the nonzero eigenvalues of $B^{-\frac{1}{2}}VB^{-\frac{1}{2}}$. Then
\[
g(k)= \ldet(B) + \sum_{i= 1}^{r} \log(1 + k \lambda_i)
\quad\mbox{ and }\quad
g^\prime(k) =\sum_{i=1}^r \frac{\lambda_i}{1 + k\lambda_i}~.
\]
Let $[r]:=\{1,\ldots,r\}$  and
$a_\ell := \left(\ell+1\right) e_{\ell+1}(\lambda_1,\ldots,\lambda_r)$, where 
\[
e_{\ell+1}(\lambda_1,\ldots,\lambda_r):=\sum_{\substack{S\subseteq [r]\,:\\ |S|=\ell+1}} \bigg( \prod_{i \in S} \lambda_i\bigg)
\]
is the $(\ell+1)$-st elementary symmetric polynomial on 
$\lambda_1,\ldots,\lambda_r$\,.
Then, the stationary point of $g$ satisfies
\begin{align}
&\sum_{\ell = 0}^{r-1}a_\ell k^\ell = 0,\label{derg0m}\\
&k\in\left(-\frac{1}{\lambda_1},-\frac{1}{\lambda_2}\right)\cup \left(-\frac{1}{\lambda_2},-\frac{1}{\lambda_3}\right)\cup\cdots\cup\left(-\frac{1}{\lambda_{r-1}},-\frac{1}{\lambda_r}\right),\label{intkm}
\end{align}
where \eqref{derg0m} determines that $g'$ vanishes at the stationary point and \eqref{intkm} ensures that the stationary point is in the domain of both $g$ and $g'$.

In the following, we establish that the matrices $V$ in which we are interested are indefinite. The result will be used in the development of closed formulas for the stationary points of $g$, when $V$ has low rank.

\begin{lemma}\label{lem:indefinite}
    Let $V\in\mathbb{S}^m$, such that $V= V^+ + V^-$, where $V^+$ is positive semidefinite with rank $r^+>0$ and $V^-$ is negative semidefinite with rank $r^->0$.     Assume that 
 $\mbox{rank}(V)= r^++r^-$. Then $V$ is indefinite and has $r^+$  positive eigenvalues and $r^-$  negative eigenvalues.
\end{lemma}

\begin{proposition}\label{cor:indef}
Let   $V^+:=\sum_{i=1}^{r^+} d_i^+v_iv_i^\top$ and 
 $V^-:=\sum_{i=1}^{r^-} d_i^-v_iv_i^\top$, where $d_i^+>0$ for all $i=1,\ldots,r^+$ and $d_i^-<0$ for all $i=1,\ldots,r^-$. Let  $V:=V^++V^-$ and assume that $\mbox{rank}(V)=r^++r^-$. Then, for any given $B\in \mathbb{S}_{++}^m$\,, $B^{-\frac{1}{2}}VB^{-\frac{1}{2}}$ is indefinite and has $r^+$ positive eigenvalues and $r^-$ negative eigenvalues.
 \end{proposition}

\subsection{Computing the optimal step size}\label{sec:intlinesearch}

\subsubsection{Rank-2 update}\label{sec:Rank-2}

We consider now $V:= d_{i_1}v_{i_1}v_{i_1}^\top +d_{i_2} v_{i_2}v_{i_2}^\top$ has rank 2 and  that $d_{i_1}+d_{i_2}=0$. 
Without loss of generality, we can take $d_{i_1}:=1$ and
$d_{i_2}:=-1$. 
We  denote by $\lambda_1,\lambda_2$ the nonzero eigenvalues of $B^{-\frac{1}{2}}VB^{-\frac{1}{2}}$.
From \eqref{derg0m} and \eqref{intkm}, we see that a stationary point $\bar k$ of $g$ satisfies the linear equation $2\lambda_1\lambda_2 \bar{k} + (\lambda_1+\lambda_2) = 0$ and $\bar{k}\in\left(-\frac{1}{\lambda_1},-\frac{1}{\lambda_2}\right)$.
Then
$
\bar{k}= - \frac{\lambda_1+\lambda_2}{\lambda_1\lambda_2}.
$

\begin{remark}
For the rank-2 update, in order to obtain $\bar{k}$ more efficiently, we use the following alternative calculation  in our implementation. We  define 
$\gamma_{ij} := v_i^\top B^{-1}v_j$.
From \eqref{invSM} and \eqref{detSM}, we have
\begin{equation*} 
g(k) = \ldet(B) + \log\left(k^2 ( \gamma_{i_1i_2}^2 -\gamma_{i_1i_1}\gamma_{i_2i_2} ) 
+k( \gamma_{i_1i_1} -\gamma_{i_2i_2}) + 1\right).
\end{equation*}
The stationary point $\bar k$ of $g$ 
should then satisfy 
$\bar k^2 (\gamma_{i_1i_2}^2 -\gamma_{i_1i_1}\gamma_{i_2i_2})  +\bar k( \gamma_{i_1i_1} -\gamma_{i_2i_2}) + 1> 0$ and 
\begin{align*}
    &g^{\prime}(\bar k) = \frac{2\bar k( \gamma_{i_1i_2}^2-\gamma_{i_1i_1}\gamma_{i_2i_2} )+ \gamma_{i_1i_1} -\gamma_{i_2i_2}}{\bar k^2 (\gamma_{i_1i_2}^2-\gamma_{i_1i_1}\gamma_{i_2i_2} )  +\bar k( \gamma_{i_1i_1} -\gamma_{i_2i_2}) + 1} = 0. 
\end{align*}
Then 
    \begin{equation}\label{kstar}
    \bar k = \frac{ \gamma_{i_1i_1} -\gamma_{i_2i_2}}{2(\gamma_{i_1i_1}\gamma_{i_2i_2}-\gamma_{i_1i_2}^2 )}~.
    \end{equation}
    We note that this last formula, with a very closely-related use, appears in \cite{Bohning}
\end{remark}

\subsubsection{Rank-3 update}\label{sec:Rank-3}

We consider that $V:= d_{i_1}v_{i_1}v_{i_1}^\top +d_{i_2} v_{i_2}v_{i_2}^\top+d_{i_3} v_{i_3} v_{i_3}^\top$ has rank 3 and  that $d_{i_1}+d_{i_2}+d_{i_3}=0$.
We note that unlike the rank-2 case, there are many possible integer directions; but our analysis below works for all of them. 
We  denote by $\lambda_1,\lambda_2,\lambda_3$ the nonzero eigenvalues of $B^{-\frac{1}{2}}VB^{-\frac{1}{2}}$. We can  assume, without loss of generality, that $\lambda_1>0> \lambda_2\geq \lambda_3$\,. 

\begin{remark}\label{rem:wlog_lam}
The ``without loss of generality'' is because of Prop. \ref{cor:indef} and that we could replace 
$d$ by $-d$ (in the definition of $V$), 
which just changes the signs of and re-labels the non-zero eigenvalues, and furthermore any stationary point $\bar{k}$ of $g$ becomes $-\bar{k}$. 
\end{remark}

From \eqref{derg0m} and \eqref{intkm}, we see that a stationary point $\bar k$ of $g$ satisfies
the quadratic equation:
\begin{align}
    &3\lambda_1\lambda_2\lambda_3 \bar{k}^2 + 2(\lambda_1\lambda_2+\lambda_1\lambda_3+\lambda_2\lambda_3) \bar{k} + \lambda_1+\lambda_2+\lambda_3 = 0,\label{derg0}\\
    &\bar{k}\in\left(-\frac{1}{\lambda_1},-\frac{1}{\lambda_2}\right)\cup \left(-\frac{1}{\lambda_2},-\frac{1}{\lambda_3}\right).\nonumber 
\end{align}

\begin{theorem}\label{prop:rank3}
  Given $\lambda_1>0>\lambda_2\geq \lambda_3$, define
     \[
    a:=
   3\lambda_1\lambda_2\lambda_3,\; \; 
    b:=2(\lambda_1\lambda_2+\lambda_1\lambda_3+\lambda_2\lambda_3),\;\;
    c:=\lambda_1+\lambda_2+\lambda_3\,.
    \]
Let
\begin{equation}\label{kpm}
\bar k_{\pm}:= 
\frac{-b \pm \sqrt{b^2 - 4ac}}{2a}\,.
\end{equation}
be the two solutions of \eqref{derg0}.
Then, 
\[
\bar{k}_{-}\in\left(-\frac{1}{\lambda_1}, -\frac{1}{\lambda_3}\right) \mbox{ and }~ \bar{k}_{+}\geq-\frac{1}{\lambda_3}.
\]
\end{theorem}

Thm. \ref{prop:rank3} implies that the unique stationary point of $g$ is $\bar{k}_-$\,.

\section{Numerical Experiments}\label{sec:num_exp}

We implemented the three local-search procedures  described in \S\ref{subsec:heur2}:  `FI', `FI$^+$', and `BI'; and two local-search directions $d$, as discussed in \S\ref{sec:LSdirections}, one defining a rank-2 update and the other defining a rank-3 update of the matrix $B:=
\sum_{\ell\in N}  \bar{x}_\ell v_\ell v_\ell^\top \in \mathbb{S}_{++}^m$, where $\bar x$ is the current solution.  The rank-2 update is given by direction  $d:= \mathbf{e}_{i_1}-\mathbf{e}_{i_2}$, for all distinct  $i_1,i_2\in N$, we experimented with step size $k=1$ 
 and with the optimal step size as discussed in \S\ref{sec:Rank-2}.
 We also used the  rank-3 update  given by direction  $d:= 2\mathbf{e}_{i_1} - \mathbf{e}_{i_2}-\mathbf{e}_{i_3}$, for all distinct  $i_1,i_2,i_3\in N$, such that $\bar x+ d$ is feasible, with optimal step size, as discussed in \S\ref{sec:Rank-3}. We note that other rank-3 updates could be used as well, but we confined our experiments to this one --- the rank-3 update corresponding to the
$d$ having minimum 1-norm. 
 We  initialized the local-search procedures using the  methods proposed in \S\ref{subsec:heur1}, namely, `Bin$(x^0)$', `Bin$(\hat{x}^0)$',  `Int$( x^0)$' and `Int$(\hat{x}^0)$'. The best solution obtained by all procedures serves as an incumbent solution for our B\&B algorithm. 
 The maximum time to compute it for an instance in our experiments was   0.15 second for \ref{prob01}  and 200 seconds for \ref{prob}. The binary case is much faster because we cannot apply  the rank-2 update with step size greater than 1 or the rank-3 update. 
 
We  implemented different versions of the B\&B algorithm, where the upper bounds are given either by the natural bound $\znatural$ or by the $\Gamma$-bound $\zgamma$\,.
We note that, for $s>m$, it would also be possible to use the upper bounds for the Data-Fusion problem \ref{datafusion} described in \S\ref{subsec:datafusionbound}, whenever the variables fixed at one at a given subproblem correspond to a rank-$m$ submatrix of $A$. 

We apply both procedures described below at each subproblem of the B\&B enumeration  in an attempt to reduce its size and make the algorithm \hbox{more efficient.}

\begin{itemize}
    \item `VBT': Compute the variable-bound tightening (VBT) inequalities \eqref{ineq1} and \eqref{ineq2} and include them  in the current subproblem, also fixing variables when possible. To compute the inequalities, we use  a dual solution for the continuous relaxation used to generate the upper bound for the current subproblem  and the best known lower bound for \ref{prob}. The dual solutions are computed by   the closed formulae 
    presented in Sections \ref{sec:natural_bound} and \ref{sec:dualgamma}.
    \item `LS': Apply the procedure `Round$(\hat x)$' 
     from the solution $\hat x$ of the continuous relaxation whenever the solution is not integer. Then apply local-search procedures from the integer solution obtained by Round$(\hat x)$, or from $\hat{x}$, in case it is already integer.  We first apply the local search `FI', and if it improves the initial solution, we also apply `FI$^+$' and `BI'. 
\end{itemize}

The algorithms proposed were coded in Julia v.1.10.0. To solve the convex relaxations \ref{cont_rel} and \ref{gamma_bound}\,, we employed Knitro 14.0.0 (Julia wrapper Knitro v0.14.2),
using 
\texttt{CONVEX = true}, for
the $\Gamma$-bound  \texttt{FEASTOL} = $10^{-9}$ (feasibility tolerance),
\texttt{OPTTOL} = $10^{-9}$ (optimality tolerance)  and for the natural bound we set \texttt{FEASTOL} = $5\cdot 10^{-7}$ (feasibility tolerance),
\texttt{OPTTOL} = $5\cdot 10^{-7}$ (optimality tolerance) ,
\texttt{ALGORITHM = 1} (Interior/Direct algorithm), 
\texttt{HESSOPT = 3} (Knitro computes a (dense) quasi-Newton SR1 Hessian). 
We prefer the strict tolerance of $10^{-9}$ that used in computing the 
$\Gamma$ bound, so as to increase the chance of identifying an integer solution; but for computing the natural bound, 
Knitro was not able to consistently converge at such a strict tolerance, so we relaxed this to $5\cdot 10^{-7}$.

To solve \ref{prob}, we employ  the B\&B algorithm in  Juniper \cite{juniper} (using the \texttt{StrongPseudoCost} branching rule,    and $10^{-5}$ as the tolerance to consider a value as integer). 
We ran our experiments on `zebratoo', a
32-core machine (running Windows Server 2022 Standard):
two Intel Xeon Gold 6444Y processors running at 3.60GHz, with 16 cores each, and 128 GB of memory.

Next, we separate  our numerical results into the binary case and the integer case. The elapsed-time limit to solve each instance in both cases was 5 hours.

\subsection{Results for \ref{prob01}}

To construct test-instances for \ref{prob01}, 
we randomly generate normally-distributed elements for the  $n\times m$ dimensional  matrices $A$ with 
rank $m$, with mean $0$ and standard deviation $1$. 
Because the variables are all binary, here we only use the rank-2 update given by direction  $d:= \mathbf{e}_{i_1}-\mathbf{e}_{i_2}$ and step size $k=1$ in our local searches. We also initialize them using only  the  methods  `Bin$(x^0)$' and `Bin$(\hat{x}^0)$' to compute the initial lower bound for the B\&B. 

Considering Remark \ref{rem:reducetoMESP},
we can reduce any instance of \ref{prob01}
to an instance of MESP, and then apply 
any upper bounding method for MESP.
 Specifically, the instance of MESP has 
$C:=I_n-UU^\top$, with the goal of choosing an order $n-s$ principle submatrix maximizing the determinant. 
An important upper bound for MESP was given by \cite{Kurt_linx}.
For $x\in[0,1]^n$ and $\gamma>0$, 
the \emph{(scaled) linx  bound} for this instance of MESP is
	\begin{align*}
	&\hypertarget{zlinxtarget}{\zlinxthing}=\max\! \left\{
	\textstyle{\frac{1}{2}}(\ldet \gamma C\Diag(y)C+\Diag(\mathbf{e}\!-\!y)-(n-s)\log \gamma) ~:~\right.\\
&\qquad\qquad\qquad\left.	 \mathbf{e}^{\top}y=n-s,~0\leq y\leq \mathbf{e} \right\}\!.
	\end{align*}
Another important upper bound for MESP is the factorization bound (see \cite{ChenFampaLee_Fact}), but this is precisely $\zgamma$\,.

In Fig. \ref{fig:compare_bounds}, we depict plots comparing  the  three bounds for \ref{prob01}, namely the natural bound, the $\Gamma$-bound, and the linx bound (for the associated MESP instance). For the linx bound, we compute the scaling parameter $\gamma$ as suggest in \cite{Kurt_linx}. In all plots, we show in the vertical axis the gap given by the difference between the upper bounds $\znatural$\,, $\zgamma$ and $\zlinx$\,, on the root subproblem of the B\&B enumeration tree, 
and the best lower bound computed by the local-search procedures.  
We illustrate in Fig. \ref{fig:compare_bounds}(a), a pattern that we have observed in the experiments with our test instances. When $n\!\lesssim\! 2m$, the $\Gamma$-bound is the best, when $n\!\gtrsim\! 2m$, the natural bound is the best, and when $n\!\approx\! 2m$, the bounds are very similar.  To better show this pattern, we set $s$ to the smallest possible value ($s\! = \! m$) in all instances in Fig. \ref{fig:compare_bounds}(a),  so we may vary $n(>\!s)$ in a larger interval. 
Note that the three plots cross when $n\approx 2m$. In Figs. \ref{fig:compare_bounds}(b,c), we can see the same behavior of the bounds for other values of $s$,  the natural bound is always the best for $100=:m\!<\! n/2$, and  the $\Gamma$-bound is always the best for $200=:m\!>\! n/2$. 

The instances of \ref{prob01} considered in Fig. \ref{fig:compare_bounds}(d) are  constructed distinctly from the other binary instances. The generation of the rows of the matrix $A$ is  based on the data ($\hat B\! \in\!\mathbb{S}_{++}^{117}$\,, $\hat G\! \in\!\mathbb{R}^{117\times 117}$)
for the 
IEEE benchmark instance with 118  buses of the phasor measurement unit (PMU) placement problem in power systems, that can be formulated as a Data-Fusion problem \ref{datafusion}. We consider the instance with small PMU standard deviation (see \cite[Sec. 5]{li2022d}). We decomposed $\hat B$ as  $\hat B:=H^\top H$ and randomly generated $G$ (in the same manner that $\hat G$ was generated in \cite{li2022d}), constructing an instance of the D-optimality problem \eqref{partdopt}, where we have all variables corresponding to the rows of $H$ fixed at one. We show in Thm. \ref{prop:comp}, that
the  $\Gamma$-bound  for \eqref{partdopt} is the same as 
the M-DDF-complementary bound $\zcg$ from \cite{li2022d}. Then, from  \cite[Fig. 4(b)]{li2022d}, we could already expect to have the $\Gamma$-bound better than the natural bound for $s=150$ (note that it corresponds to $s=33$ in \cite[Fig. 4(b)]{li2022d}). Our goal with the experiment reported in Fig. \ref{fig:compare_bounds}(d), is to compare our more general $\Gamma$-bound and  the natural bound  for the instances derived from \eqref{partdopt}, and the linx bound for the associated  MESP instance, but now unfixing the variables, one by one. Note that  we could not apply the M-DDF-complementary bound once variables are unfixed,
as it requires $H$ to be full column-rank.  
In constructing Fig. \ref{fig:compare_bounds}(d), we applied a monotone function,
specifically the fourth root, to the gaps, so
that we can more easily see the behavior when the gaps are both small and large. 
We see that the compared bounds are all close to zero when the number of fixed variables is up to 27, but when we increase the number from there, the $\Gamma$-bound becomes more and more superior. When we have 117 variables fixed, our $\Gamma$-bound becomes equivalent to the  M-DDF-complementary bound. With the experiment, we demonstrate that the generalization of the M-DDF-complementary bound that we propose for \ref{prob01} maintains its superiority with respect to the natural bound and to the linx bound. 

\begin{figure}[!ht]%
    \centering
    \subfloat[]{{\includegraphics[scale=0.25]{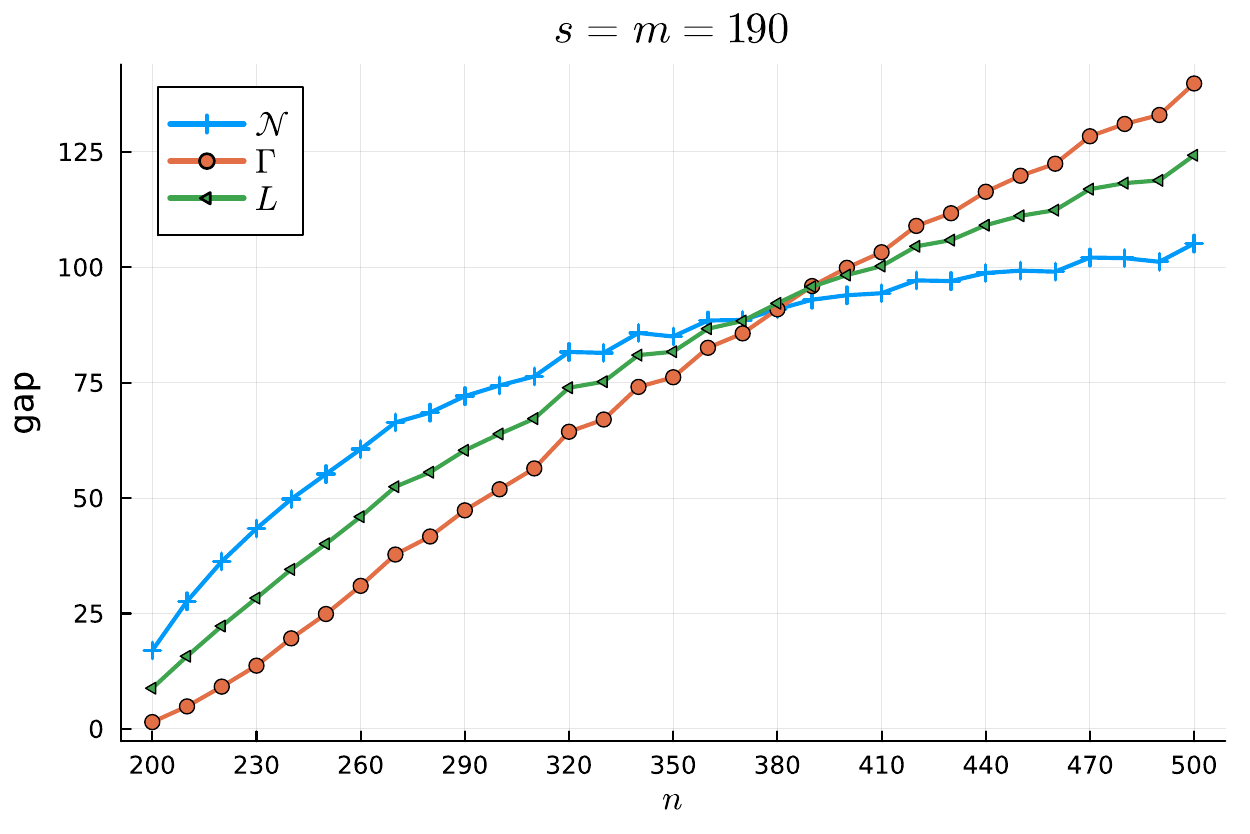} }}%
    ~
    \subfloat[]{{\includegraphics[scale=0.25]{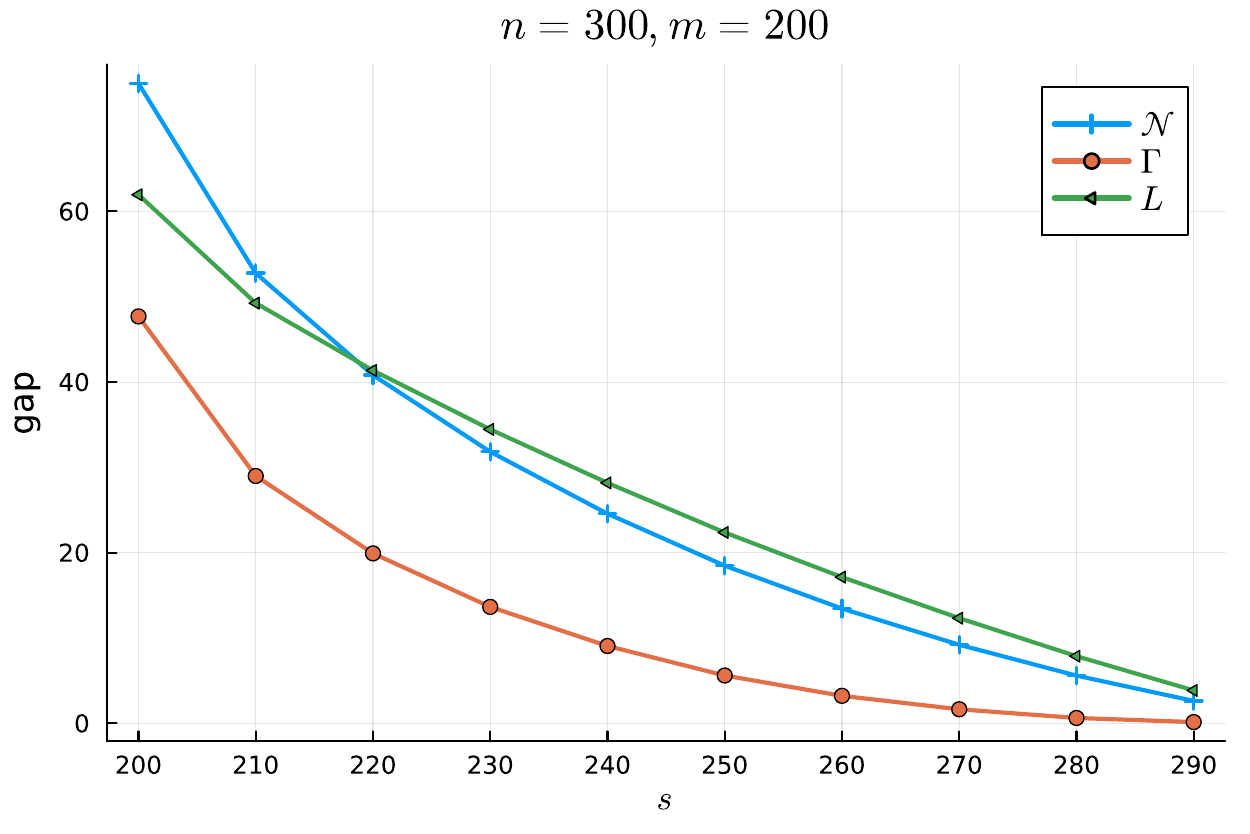} }}%
    \hfill
    \subfloat[]{{\includegraphics[scale=0.25]{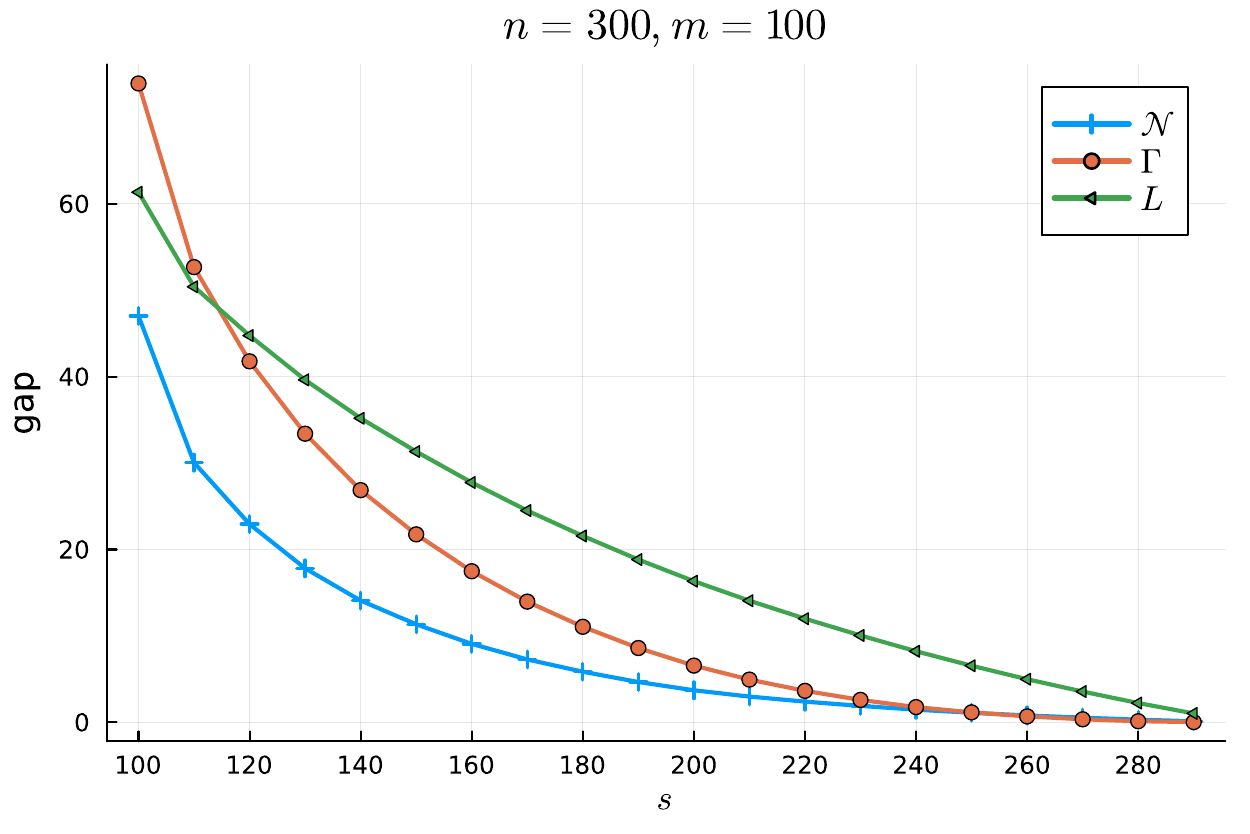} }}%
    ~
    \subfloat[]{{\includegraphics[scale=0.25]{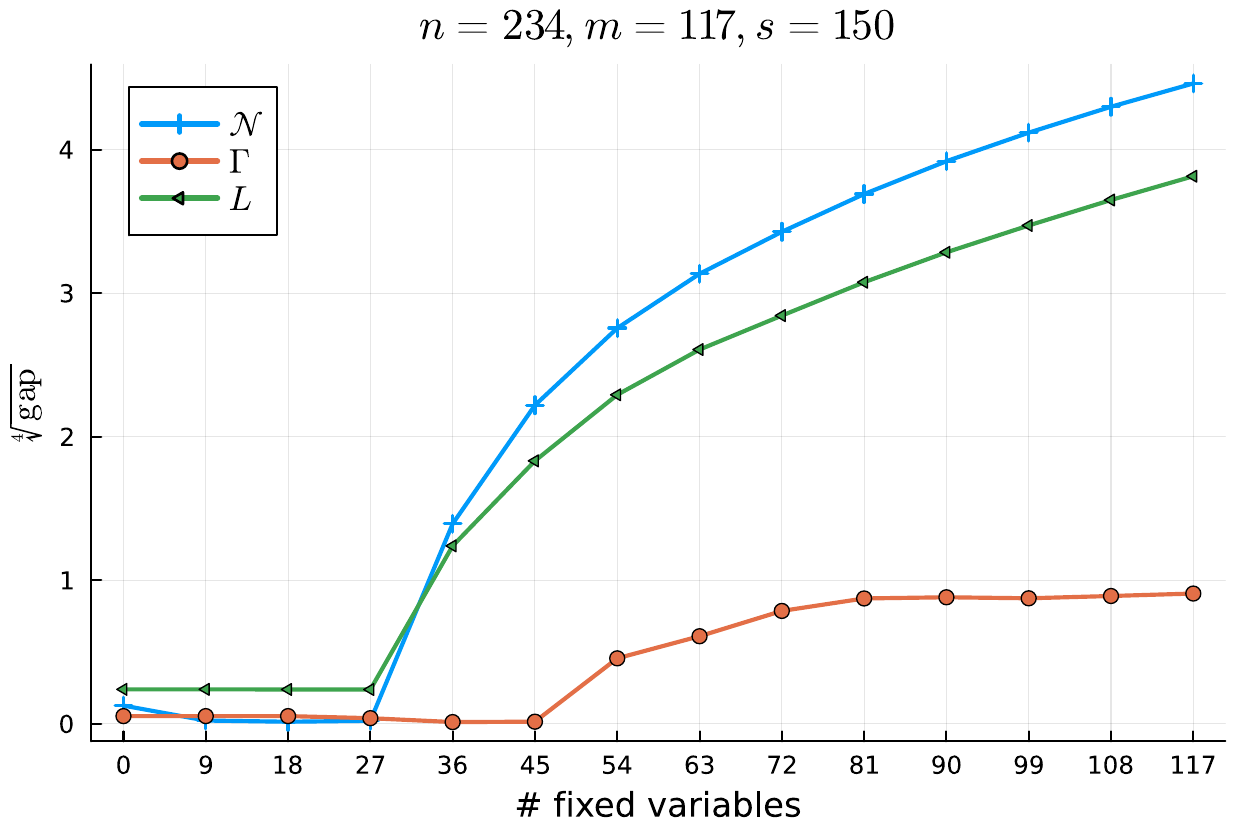} }}%
     \caption{Comparing the bounds for \ref{prob01}}%
    \label{fig:compare_bounds}%
\end{figure}

In Table \ref{tab:bb_bin_onebound}, we show statistics obtained when running the B\&B based either on the natural bound $\znatural$ or on the $\Gamma$-bound $\zgamma$\,.
 For each instance, we show from the left to the right, the parameters $n,m,s$;  
the `root gap', given by the difference between the upper bound on the root subproblem of the B\&B enumeration tree and the lower bound given by the best known solution for the instance, sometimes proved optimal; the `final gap', given by the difference between the final upper bound after running the B\&B and the same lower bound;  the elapsed `time' in seconds (`$*$' represents that the time limit was reached); the number of subproblems (`\raisebox{1.2pt}{{\scriptsize{$\#$}}}\! nodes'), that is, the number of subproblems solved; the number of  variables fixed by the VBT inequalities (`\raisebox{1.2pt}{{\scriptsize{$\#$}}}\! fixed variables'); 
and the number of times the local searches `LS' improve the incumbent solution during the execution of the B\&B (`\raisebox{1.2pt}{{\scriptsize{$\#$}}}\! LSI').  We note that besides showing the root gap for the two bounds used inside the B\&B, we also present the root gap for the linx bound ($L$).
Once more we see that the linx bound is always dominated by one of the other bounds even using the scaling parameter suggested in \cite{Kurt_linx}, and therefore, we did not apply it inside the B\&B. Also, it is important to note that the final (rigorous) gaps that
we report are based on dual-feasible solutions that we 
compute ourselves, while Juniper terminates when its
``mip gap'' based on an approximate primal-optimal solution is essentially zero. 

We can see that when $m=0.25n$, the natural bound gives  better results, when $m=0.5n$ the results are very similar for all bounds and when $m=0.75n$, the $\Gamma$-bound gives better results. The results confirm the analysis derived from Fig. \ref{fig:compare_bounds}(a--c).
 We  note that when the time limit for the B\&B is reached for both bounds (all instances with $n=60$, for example),  the final gap is much smaller for the natural bound when $m=0.25n$ and the final gap is much smaller for the $\Gamma$-bound when $m=0.75n$. 
 For the instance solved to optimality by both methods, with very similar bounds at the root subproblem ($n=40,m=20$), we see that both the number of subproblems  and the elapsed time are smaller  for the $\Gamma$-bound, showing that it was stronger than the natural bounds on the subproblems. For instances solved to optimality  with $m=0.5n$ ($n=20,30,40$), for which the two bounds are competitive at the root subproblem, we see that more variables are fixed when more subproblems are evaluated, showing how effective the VBT inequalities are for these instances. On the other hand, the local-search procedure only improves the current best solution a few times. Nevertheless, as we will see in the Table \ref{tab:feature_bin}, the procedure still has a good overall impact on the B\&B. Concerning the number of variables fixed in  columns ($\mathcal{N}$ and $\Gamma$), we note that they are, in fact, the number of  variables fixed in $\mathcal{N}_{{\mbox{\protect\tiny D-Opt}}}$ and $\Gamma_{{\mbox{\protect\tiny D-Opt(0/1)}}}$, respectively, and also that they are all fixed at zero. Therefore, in column $\mathcal{N}$, we show the number of variables in D-Opt(0/1) that are fixed at zero, and in column $\Gamma$, we show the number of variables in D-Opt(0/1) that are fixed at one (because the variable $y$ in $\Gamma_{{\mbox{\protect\tiny D-Opt(0/1)}}}$ is the complement of the variable $x$ in $\mathcal{N}_{{\mbox{\protect\tiny D-Opt}}}$). We note that fixing is driven by quality of  bounds and also by sparsity in relaxations.  When $s$ is small (in the table $n/4$), the natural bound is better than the $\Gamma$-bound and when $s$ is large (in the table $3n/4$), the $\Gamma$-bound is better than the natural bound. 
For small $s$, when solving $\mathcal{N}_{{\mbox{\protect\tiny D-Opt}}}$ we get many $x$ variables at value zero and relatively few positive (and so few at value one); so we see a considerable amount of fixing for these instances, and only fixing at zero.  By the same reasoning, for large $s$, when solving $\Gamma_{{\mbox{\protect\tiny D-Opt(0/1)}}}$\,,  
many $y$ variables are fixed at zero, which translates to many $x$ variables fixed at one. 
For the between cases (in the table $n/2$),
both bounds offer some significant fixing of $x$ variables, 
 with some variables fixed at zero and some at one.

\begin{table}[!ht]
\scriptsize
\centering
\begin{tabular}{l|rrr|rr|rr|rr|rr|rr}
\multicolumn{1}{c|}{$n,m$$\vphantom{\Sigma^{I^I}}$} & \multicolumn{3}{c|}{\scriptsize root gap}                                                   & \multicolumn{2}{c|}{\scriptsize gap}                              & \multicolumn{2}{c|}{\scriptsize time (sec)}                       & \multicolumn{2}{c|}{\raisebox{1.2pt}{{\scriptsize{$\#$}}}\! nodes} & \multicolumn{2}{c|}{\raisebox{1.2pt}{{\scriptsize{$\#$}}}\! fixed variables} & \multicolumn{2}{c}{\raisebox{1.2pt}{{\scriptsize{$\#$}}}\! LSI}  \\
\multicolumn{1}{c|}{\scriptsize $s\!=\! m$}         & \multicolumn{1}{c}{$\mathcal{N}$} & \multicolumn{1}{c}{$\Gamma$} & \multicolumn{1}{c|}{$L$} & \multicolumn{1}{c}{$\mathcal{N}$} & \multicolumn{1}{c|}{$\Gamma$} & \multicolumn{1}{c}{$\mathcal{N}$} & \multicolumn{1}{c|}{$\Gamma$} & \multicolumn{1}{c}{$\mathcal{N}$}  & \multicolumn{1}{c|}{$\Gamma$} & \multicolumn{1}{c}{$\mathcal{N}$}       & \multicolumn{1}{c|}{$\Gamma$}      & \multicolumn{1}{c}{$\mathcal{N}$} & \multicolumn{1}{c}{$\Gamma$} \\ \hline
20,5                                                & 0.15                              & 1.67                         & 0.87                     & 0.00                              & 0.00                          & 5.3                               & 7.8                           & 36                                 & 465                           & 136                                     & 1                                  & 0                                 & 0                            \\
20,10                                               & 0.83                              & 0.80                         & 0.80                     & 0.00                              & 0.00                          & 1.0                               & 0.8                           & 126                                & 154                           & 32                                      & 39                                 & 1                                 & 1                            \\
20,15                                               & 1.00                              & 0.21                         & 0.58                     & 0.00                              & 0.00                          & 0.9                               & 0.2                           & 119                                & 32                            & 0                                       & 139                                & 0                                 & 0                            \\ \hline
30,7                                                & 0.43                              & 2.72                         & 1.50                     & 0.00                              & 0.00                          & 0.7                               & 24.7                          & 83                                 & 4537                          & 134                                     & 12                                 & 1                                 & 1                            \\
30,15                                               & 2.10                              & 1.99                         & 2.07                     & 0.00                              & 0.00                          & 74.0                              & 40.1                          & 8562                               & 7458                          & 2180                                    & 1937                               & 1                                 & 1                            \\
30,22                                               & 2.57                              & 0.92                         & 1.67                     & 0.01                              & 0.00                          & 33.9                              & 1.3                           & 3792                               & 256                           & 225                                     & 219                                & 1                                 & 0                            \\ \hline
40,10                                               & 1.03                              & 4.06                         & 2.50                     & 0.00                              & 0.00                          & 14.8                              & 833.4                         & 1742                               & 131917                        & 1384                                    & 5741                               & 1                                 & 1                            \\
40,20                                               & 2.92                              & 2.93                         & 2.96                     & 0.01                              & 0.01                          & 1974.4                            & 1012.4                        & 173700                             & 159921                        & 70533                                   & 60676                              & 1                                 & 1                            \\
40,30                                               & 4.26                              & 1.11                         & 2.59                     & 0.01                              & 0.01                          & 2339.0                            & 7.9                           & 222288                             & 1712                          & 451                                     & 1530                               & 1                                 & 3                            \\ \hline
50,12                                               & 1.12                              & 5.32                         & 3.04                     & 0.00                              & 0.62                          & 16.3                              & *                             & 1780                               & 925109                        & 2498                                    & 4739                               & 1                                 & 2                            \\
50,25                                               & 4.23                              & 4.29                         & 4.31                     & 1.13                              & 1.07                          & *                                 & *                             & 671867                             & 835819                        & 27851                                   & 52773                              & 1                                 & 1                            \\
50,37                                               & 5.30                              & 1.44                         & 3.22                     & 0.85                              & 0.01                          & *                                 & 33.6                          & 619813                             & 5929                          & 2343                                    & 7525                               & 0                                 & 0                            \\ \hline
60,15                                               & 2.72                              & 8.60                         & 5.60                     & 0.72                              & 4.33                          & *                                 & *                             & 674021                             & 797115                        & 1706746                                 & 0                                  & 2                                 & 1                            \\
60,30                                               & 6.02                              & 5.83                         & 6.04                     & 3.04                              & 2.80                          & *                                 & *                             & 624285                             & 784199                        & 0                                       & 0                                  & 1                                 & 1                            \\
60,45                                               & 7.89                              & 2.44                         & 5.03                     & 3.68                              & 0.11                          & *                                 & *                             & 600631                             & 1059441                       & 0                                       & 2228236                            & 1                                 & 1                           
\end{tabular}
\caption{B\&B for \ref{prob01} using either the $\Gamma$-bound or the natural bound}\label{tab:bb_bin_onebound}
\end{table}

In Table \ref{tab:feature_bin}, to analyse the impact of our enhancement procedures on the B\&B algorithm, we show  results for different versions of the B\&B, when applied to the two most difficult instances from Table \ref{tab:bb_bin_onebound}, for which the natural bound and the $\Gamma$-bound  are competitive (with  $(n,m)$ given by $(50,25)$ and $(60,30)$).
 Neither instance could  be solved to optimality by any version of the B\&B in the time limit. Each row in the table corresponds to a version of the B\&B, where the procedures used to enhance it are identified in the first column. In the other columns, we  present results for the two versions of the  B\&B, that are based either on the natural bound (`$\mathcal{N}$') or on the $\Gamma$-bound (`$\Gamma$').   We show  the final gaps obtained  and number of subproblems in the B\&B tree,  for  each instance. We see that the VBT inequalities often increase  the number of subproblems solved in the time limit, showing that tightening variable bounds is usually effective in speeding up the resolution of subproblems. On the other hand, the number of subproblems decreases when we add  `LS' to `VBT', showing the overhead  introduced by the local search. Nevertheless, we see that `LS' is important to obtain the best gaps for the natural bound  and both procedures are important  to obtain  the smallest gaps for the $\Gamma$-bound these instances. 

\begin{table}[!ht]
\centering
\begin{tabular}{l|rrrr|rrrr}
           & \multicolumn{4}{c|}{gap}                                                                                                                  & \multicolumn{4}{c}{\# nodes}                                                                                                             \\ \cline{2-9} 
procedures & \multicolumn{2}{c|}{50,25}                                          & \multicolumn{2}{c|}{60,30}                                          & \multicolumn{2}{c|}{50,25}                                          & \multicolumn{2}{c}{60,30}                                          \\
in B\&B    & \multicolumn{1}{c}{$\mathcal{N}$} & \multicolumn{1}{c|}{${\Gamma}$} & \multicolumn{1}{c}{$\mathcal{N}$} & \multicolumn{1}{c|}{${\Gamma}$} & \multicolumn{1}{c}{$\mathcal{N}$} & \multicolumn{1}{c|}{${\Gamma}$} & \multicolumn{1}{c}{$\mathcal{N}$} & \multicolumn{1}{c}{${\Gamma}$} \\ \hline
none       & 1.132                             & \multicolumn{1}{r|}{1.114}      & 3.680                             & 3.481                           & 698617                            & \multicolumn{1}{r|}{769775}     & 679901                            & 751231                         \\
VBT        & 1.132                             & \multicolumn{1}{r|}{1.114}      & 3.680                             & 3.479                           & 698961                            & \multicolumn{1}{r|}{769373}     & 680225                            & 757407                         \\
LS         & 1.127                             & \multicolumn{1}{r|}{1.107}      & 3.034                             & 2.833                           & 673999                            & \multicolumn{1}{r|}{746061}     & 629847                            & 699261                         \\
VBT+LS     & 1.127                             & \multicolumn{1}{r|}{1.072}      & 3.036                             & 2.802                           & 671867                            & \multicolumn{1}{r|}{835819}     & 624285                            & 784199                        
\end{tabular}
\caption{Effect of enhancements on B\&B}
\label{tab:feature_bin}
\end{table}

 We constructed other instances of \ref{prob01}
 based on some real data: the ``Regensburg Pediatric Appendicitis'' data set at the
 University of California Irvine (UCI) Machine Learning Repository; see \url{https://archive.ics.uci.edu/dataset/938/regensburg+pediatric+appendicitis} and
 \cite{Appendicitis}. 
 We took the rows corresponding to females 
between the ages of 11 and 14, deleting columns with more than 
40 values missing, and then deleting rows with any missing values.
For the categorical variables,  we used ``one-hot encoding''. 
This left us with 71 rows and 40 columns.  

 We show in Table \ref{tab:appendicitis}
results obtained for $s\in\{40,45,\ldots,65\}$ and we see that the $\Gamma$-bound is superior to the natural bound and to the linx bound at the root subproblem of the B\&B, for all values of $s$. Thus, we apply the B\&B based on the $\Gamma$-bound for these instances. The difficulty of the problems decreases as we increase $s$. We could solve the instance to optimality for all  $s>50$, and we see a significant reduction on the gap for the smaller $s$.

\begin{table}[!ht]
\centering
\scriptsize
\begin{tabular}{l|rrr|r|r|r}
                                     & \multicolumn{3}{c|}{\scriptsize root gap}                                                   & \multicolumn{1}{c|}{\scriptsize gap} & \multicolumn{1}{c|}{\scriptsize time (sec)} & \multicolumn{1}{c}{\raisebox{1.2pt}{{\scriptsize{$\#$}}}\! nodes}   \\
\multicolumn{1}{c|}{\scriptsize $s$} & \multicolumn{1}{c}{$\mathcal{N}$} & \multicolumn{1}{c}{$\Gamma$} & \multicolumn{1}{c|}{$L$} & \multicolumn{1}{c|}{$\Gamma$}        & \multicolumn{1}{c|}{$\Gamma$}               & \multicolumn{1}{c}{$\Gamma$}                                                                        \\ \hline
40                                   & 7.90                              & 6.55                         & 7.33                     & 3.72                                 & *                                           & 653817                                                                                                                           \\
45                                   & 4.25                              & 3.63                         & 6.24                     & 1.98                                 & *                                           & 661791                                                                                                                           \\
50                                   & 2.24                              & 1.79                         & 4.40                     & 0.56                                 & *                                           & 684555                                                                                                                           \\
55                                   & 1.19                              & 0.90                         & 2.94                     & 0.01                                 & 2353.3                                      & 271436                                                                                                                            \\
60                                   & 0.52                              & 0.42                         & 1.67                     & 0.01                                 & 10.3                                        & 1514                                                                                                                              \\
65                                   & 0.13                              & 0.11                         & 0.66                     & 0.01                                 & 0.3                                         & 43                                                                                                                               
\end{tabular}
\caption{Regensburg pediatric-appendicitis instances}
\label{tab:appendicitis} 
\end{table}

We constructed another set of instances, meant as a challenging combinatorial instance with symmetry.
For $t>1$, we make a $\binom{t}{2}\times t$ matrix that is the transpose of the $0/\pm1$-valued node-arc incidence matrix of an orientation of the edges of a complete undirected graph on $t$ vertices. 
The rank of such a matrix is well known to be $t-1$, and removing any column, which we do,
gives us a $\binom{t}{2}\times (t-1)$ 
matrix with rank $t-1$. The \ref{prob01} problem for such a matrix aims to select $s$ edges so as to maximize the number of spanning trees for that chosen subgraph
(this directly follows from the Cauchy-Binet formula and total unimodularity of node-arc incidence matrices). Our results are in Table \ref{tab:graphs}.
Because of the symmetries of the complete graph induced by permutations of its vertices, 
typically we will see a huge number of alternative optimal solutions, making these very hard instances for a B\&B that seeks to prove optimality, despite the good quality of the natural bound, for which we see root gaps less than one for most instances. In our experiments, we set $t=20$, and so our matrix is 
$190\times 19$. 

\begin{table}[!ht]
\centering
\scriptsize
\begin{tabular}{l|rrr|r|r}
\multicolumn{1}{c|}{}                & \multicolumn{3}{c|}{\scriptsize root gap}                                                   & \multicolumn{1}{c|}{\scriptsize gap} & \multicolumn{1}{c}{\raisebox{1.2pt}{{\scriptsize{$\#$}}}\! nodes} \\
\multicolumn{1}{c|}{\scriptsize $s$} & \multicolumn{1}{c}{$\mathcal{N}$} & \multicolumn{1}{c}{$\Gamma$} & \multicolumn{1}{c|}{$L$} & \multicolumn{1}{c|}{$\mathcal{N}$}   & \multicolumn{1}{c}{$\mathcal{N}$}                                 \\ \hline
19                                   & 10.17                             & 35.91                        & 23.43                    & 9.50                                 & 283245                                                            \\
38                                   & 2.19                              & 16.76                        & 15.06                    & 1.99                                 & 154523                                                            \\
57                                   & 1.09                              & 9.95                         & 11.22                    & 0.99                                 & 75083                                                             \\
76                                   & 0.63                              & 6.03                         & 8.53                     & 0.58                                 & 52649                                                             \\
95                                   & 0.36                              & 3.52                         & 6.42                     & 0.32                                 & 47231                                                             \\
114                                  & 0.27                              & 1.97                         & 4.77                     & 0.24                                 & 54083                                                             \\
133                                  & 0.18                              & 0.95                         & 3.34                     & 0.16                                 & 62407                                                             \\
152                                  & 0.11                              & 0.35                         & 2.10                     & 0.10                                 & 95013                                                             \\
171                                  & 0.05                              & 0.05                         & 1.00                     & 0.05                                 & 186513                                                           
\end{tabular}
\caption{Graph instances}\label{tab:graphs}
\end{table}

\subsection{Results for \ref{prob}}
To construct test instances for \ref{prob}, we used the Matlab function \texttt{sprand} to  generate random $n\times m$-dimensional  matrices $A$ with 
rank $m$, and the Matlab function \texttt{randi} to generate  random $n$-dimensional vectors
$u$, with integer components between 1 and 10. We consider $m=0.25 n$,  $s\in\{0.50n,0.75n,n\}$, and $l=0$.

In Table \ref{tab:results_integer}, we show statistics obtained when running the B\&B for \ref{prob}, based on the natural bound $\znatural$\,. We also experimented with the  $\Gamma$-bound $\zgamma$ computed by reformulating \ref{prob} as \ref{prob01}
 (see Remark \ref{rem:gammafordopt}), however the results obtained with the natural bound were superior. The superiority of the natural bound is coherent with the results observed for \ref{prob01}, because as we repeat rows of $A$ in the reformulation of \ref{prob},  we end up with $m\ll 0.5 n$ in all instances. 

\begin{table}[!ht]
\centering
\begin{tabular}{c|r|r|r|r|r|r|r|rr}
\scriptsize $n,m,s$ & \multicolumn{1}{c|}{\scriptsize root gap} & \multicolumn{1}{c|}{\scriptsize gap} & \multicolumn{1}{c|}{\scriptsize time (sec)} & \multicolumn{1}{c|}{\scriptsize \# nodes} & \multicolumn{1}{c|}{\scriptsize \# cuts} & \multicolumn{1}{c|}{\scriptsize \# fixed} & \multicolumn{1}{c|}{\scriptsize \# LSI} & \multicolumn{1}{c}{\scriptsize \# k\_bin} & \multicolumn{1}{c}{\scriptsize \# k\_int} \\ \hline
20,5,10             & 0.029                                     & 0.001                                & 7                                           & 121                                       & 940                                      & 146                                       & 0                                       & 12                                        & 0                                         \\
40,10,20            & 0.030                                     & 0.000                                & 6                                           & 139                                       & 3106                                     & 404                                       & 0                                       & 17                                        & 1                                         \\
60,15,30            & 0.029                                     & 0.002                                & 26                                          & 407                                       & 4649                                     & 734                                       & 1                                       & 108                                       & 0                                         \\
80,20,40            & 0.130                                     & 0.002                                & 1893                                        & 14577                                     & 73312                                    & 20053                                     & 3                                       & 10469                                     & 1                                         \\
100,25,50           & 0.162                                     & 0.002                                & 16085                                       & 157916                                    & 391535                                   & 137283                                    & 2                                       & 95294                                     & 26                                        \\
120,30,60           & 0.044                                     & 0.002                                & 146                                         & 327                                       & 8665                                     & 724                                       & 1                                       & 119                                       & 0                                         \\
140,35,70           & 0.215                                     & 0.099                                & *                                           & 23963                                     & 322460                                   & 77137                                     & 7                                       & 23817                                     & 5                                         \\
160,40,80           & 0.069                                     & 0.013                                & *                                           & 17219                                     & 100813                                   & 30471                                     & 3                                       & 17146                                     & 2                                         \\
180,45,90           & 0.230                                     & 0.151                                & *                                           & 10421                                     & 72156                                    & 2233                                      & 6                                       & 10358                                     & 9                                         \\
200,50,100          & 0.121                                     & 0.077                                & *                                           & 11911                                     & 168641                                   & 21791                                     & 5                                       & 11855                                     & 2   \\
\hline
20,5,15             & 0.024                                     & 0.001                                & 10                                          & 553                                       & 1197                                     & 275                                       & 0                                       & 192                                       & 0                                         \\
40,10,30            & 0.018                                     & 0.000                                & 13                                          & 877                                       & 3972                                     & 744                                       & 0                                       & 371                                       & 15                                        \\
60,15,45            & 0.023                                     & 0.002                                & 16                                          & 747                                       & 5072                                     & 1103                                      & 2                                       & 308                                       & 0                                         \\
80,20,60            & 0.084                                     & 0.003                                & 7100                                        & 143707                                    & 652911                                   & 233928                                    & 3                                       & 104183                                    & 309                                       \\
100,25,75           & 0.083                                     & 0.001                                & 12291                                       & 305767                                    & 857581                                   & 357241                                    & 2                                       & 227691                                    & 338                                       \\
120,30,90           & 0.029                                     & 0.001                                & 699                                         & 6063                                      & 26660                                    & 6592                                      & 3                                       & 2387                                      & 8                                         \\
140,35,105          & 0.107                                     & 0.037                                & *                                           & 36457                                     & 449310                                   & 120977                                    & 3                                       & 36188                                     & 336                                       \\
160,40,120          & 0.047                                     & 0.022                                & *                                           & 16093                                     & 261428                                   & 55565                                     & 12                                      & 16022                                     & 0                                         \\
180,45,135          & 0.120                                     & 0.086                                & *                                           & 11183                                     & 65904                                    & 1874                                      & 3                                       & 11126                                     & 30                                        \\
200,50,150          & 0.077                                     & 0.053                                & *                                           & 9037                                      & 50559                                    & 4346                                      & 8                                       & 8990                                      & 17       \\
\hline
20,5,20             & 0.015                                     & 0.001                                & 14                                          & 791                                       & 1281                                     & 350                                       & 0                                       & 274                                       & 0                                         \\
40,10,40            & 0.010                                     & 0.000                                & 23                                          & 733                                       & 3998                                     & 729                                       & 0                                       & 157                                       & 6                                         \\
60,15,60            & 0.008                                     & 0.003                                & 11                                          & 83                                        & 4136                                     & 666                                       & 1                                       & 12                                        & 0                                         \\
80,20,80            & 0.041                                     & 0.004                                & 3103                                        & 21257                                     & 103294                                   & 41154                                     & 3                                       & 15318                                     & 14                                        \\
100,25,100          & 0.058                                     & 0.008                                & *                                           & 86743                                     & 659840                                   & 264737                                    & 2                                       & 84706                                     & 2017                                      \\
120,30,120          & 0.027                                     & 0.001                                & *                                           & 108211                                    & 147450                                   & 51084                                     & 2                                       & 57778                                     & 17                                        \\
140,35,140          & 0.060                                     & 0.028                                & *                                           & 19739                                     & 245671                                   & 49117                                     & 3                                       & 19574                                     & 221                                       \\
160,40,160          & 0.034                                     & 0.018                                & *                                           & 16165                                     & 228384                                   & 53815                                     & 7                                       & 16092                                     & 9                                         \\
180,45,180          & 0.073                                     & 0.053                                & *                                           & 10493                                     & 50966                                    & 2548                                      & 7                                       & 10424                                     & 2                                         \\
200,50,200          & 0.053                                     & 0.040                                & *                                           & 9023                                      & 55572                                    & 3683                                      & 7                                       & 8974                                      & 12                                       
\end{tabular}
\caption{B\&B for \ref{prob} using  the natural bound}\label{tab:results_integer}
\end{table}

Besides the statistics shown in Table \ref{tab:bb_bin_onebound} for \ref{prob01}, we also show in Table \ref{tab:results_integer},  the number of  times VBT inequalities  tightened the bounds of the variables (`\raisebox{1.2pt}{{\scriptsize{$\#$}}}\! cuts') -- this number includes the number of variables fixed by the VBT inequalities (`\raisebox{1.2pt}{{\scriptsize{$\#$}}}\! fixed'), and the number of times the local search procedures improve its initial solution using the rank-2 update with step size $|k|=1$  (`\raisebox{1.2pt}{{\scriptsize{$\#$}}}\! k\_bin') and with step size $|k|>1$  (`\raisebox{1.2pt}{{\scriptsize{$\#$}}}\! k\_int'). 
We do not apply the rank-3 update inside the B\&B procedure because it is too costly.

 In the time limit, we were able to solve larger instances (with $n$ up to 120)  of \ref{prob} than of \ref{prob01}, as reported in Table \ref{tab:bb_bin_onebound}. Moreover, for the instances not solved to optimality, we got smaller gaps than the ones reported for the non solved instances of \ref{prob01}. We see that the VBT inequalities are effective for tightening the bounds of the variables and also for fixing them. Finally, we note that although the step size $|k|=1$ is optimal in most cases for improving the initial solution of the local search procedures, we also have cases where  $|k|>1$ is the best option. 

\section{Conclusion}
Our numerical experiments indicate promising directions to investigate in order to improve the efficiency of our B\&B algorithm to solve  \ref{prob}. One possible approach,  for example, is to use  different bounds for subproblems. We experimented with a version of our B\&B algorithm where both bounds, the natural bound and the $\Gamma$-bound, were computed at every subproblem, and the best one was used. For instances of \ref{prob01} with $m=0.5n$, for which both bounds were very competitive at the root subproblem of the B\&B tree, the total number of subproblems required to solve the instances to optimality  decreased with respect to the B\&B algorithms based on each bound. However, computing both bounds at every subproblem was too expensive and the overall resolution time increased. An alternative would be to compute both bounds only at every $p$ subproblems, for a given $p$, and pick the best bound at the subproblem to be the bound computed for all subproblems from that point. This is a standard approach to try to accelerate the convergence of  B\&B algorithms, and, in fact, all of the bounds for the Data-Fusion problem could also be taken into consideration once the variables fixed at value 1 at a subproblem select a full-rank submatrix of the input matrix $A$. It is interesting to note that during the experiments with our B\&B algorithm, where both the natural bound and the $\Gamma$-bound were computed at every subproblem, we observed that when one bound  was much better than the other at the root subproblem of the B\&B tree, then the same bound was the best one for the great majority of the subproblems solved. 

With regard to our analysis 
of integer line searches for rank 2 and 3 (in \S\ref{sec:intlinesearch}),
we could go further, investigating higher-rank matrices $V$; in particular 
ranks 4 and 5 could be amenable to closed-form solution
(due to the solvability of cubics and quartics by radicals). 
We leave this as well as more sophisticated ideas for 
finding good search directions for future research.

\section*{Acknowledgments}

We thank the three anonymous referees for their comments, which greatly improved our presentation.

 G. Ponte was supported in part by CNPq GM-GD scholarship 161501/2022-2. M. Fampa was supported in part by CNPq grants 305444/2019-0 and 434683/2018-3.  J. Lee was supported in part by AFOSR grant FA9550-22-1-0172. This work is partially based on work supported by the 
National Science Foundation under Grant DMS-1929284 while 
the authors were in residence at the Institute for Computational and Experimental Research in 
Mathematics (ICERM) in Providence, RI, during the Discrete Optimization program.

\bibliography{Dopt_FLP}

\begin{thebibliography}{10}

\bibitem{Ahipasaoglu2021}
Selin~Damla Ahipa\c{s}ao\v{g}lu.
\newblock A branch-and-bound algorithm for the exact optimal experimental design problem.
\newblock {\em Statistics and Computing}, 31(65), 2021.

\bibitem{Anstreicher_BQP_entropy}
Kurt~M. Anstreicher.
\newblock {Maximum-entropy sampling and the Boolean quadric polytope}.
\newblock {\em Journal of Global Optimization}, 72(4):603--618, 2018.

\bibitem{Kurt_linx}
Kurt~M. Anstreicher.
\newblock Efficient solution of maximum-entropy sampling problems.
\newblock {\em Operations Research}, 68(6):1826--1835, 2020.

\bibitem{AFLW_Using}
{Kurt M.} Anstreicher, Marcia Fampa, Jon Lee, and Joy Williams.
\newblock Using continuous nonlinear relaxations to solve constrained maximum-entropy sampling problems.
\newblock {\em Mathematical Programming}, 85:221--240, 1999.

\bibitem{Atkinson}
{Anthony C.} Atkinson, {Alexander N.} Donev, and {Randall D.} Tobias.
\newblock {\em Optimum Experimental Designs, with SAS}.
\newblock Oxford Statistical Science Series. Oxford University Press, United Kingdom, 2007.

\bibitem{Bohning}
Dankmar Böhning.
\newblock A vertex-exchange-method in {D}-optimal design theory.
\newblock {\em Metrika}, 33:337--347, 1986.

\bibitem{CaseltonZidek1984}
William~F. Caselton and James~V. Zidek.
\newblock Optimal monitoring network design.
\newblock {\em Statistics and Probability Letters}, 2:223--227, 1984.

\bibitem{ChenFampaLee_Fact}
Zhongzhu Chen, Marcia Fampa, and Jon Lee.
\newblock On computing with some convex relaxations for the maximum-entropy sampling problem.
\newblock {\em INFORMS Journal on Computing}, 35:368--385, 2023.

\bibitem{FL2022}
Marcia Fampa and Jon Lee.
\newblock {\em Maximum-Entropy Sampling: Algorithms and Application}.
\newblock Springer, 2022.

\bibitem{FLPX2021}
Marcia Fampa, Jon Lee, Gabriel Ponte, and Luze Xu.
\newblock Experimental analysis of local searches for sparse reflexive generalized inverses.
\newblock {\em Journal of Global Optimization}, 81:1057--1093, 2021.

\bibitem{Fedorov}
Valerii~V. Fedorov.
\newblock {\em Theory of Optimal Experiments}.
\newblock Academic Press, New York-London, 1972.
\newblock Translated from the Russian and edited by W. J. Studden and E. M. Klimko.

\bibitem{GVL1996}
Gene~H. Golub and Charles~F. Van~Loan.
\newblock {\em Matrix Computations (3rd Ed.)}.
\newblock Johns Hopkins University Press, Baltimore, MD, USA, 1996.

\bibitem{HARMAN2020}
Radoslav Harman and Samuel Rosa.
\newblock On greedy heuristics for computing {D}-efficient saturated subsets.
\newblock {\em Operations Research Letters}, 48(2):122--129, 2020.

\bibitem{HJBook}
Roger~A. Horn and Charles~R. Johnson.
\newblock {\em Matrix Analysis}.
\newblock Cambridge University Press, Cambridge, {F}irst edition, 1985.

\bibitem{Kiefer}
Jack Kiefer.
\newblock On the nonrandomized optimality and randomized nonoptimality of symmetrical designs.
\newblock {\em The Annals of Mathematical Statistics}, 29(3):675--699, 1958.

\bibitem{KW}
Jack Kiefer and Jacob Wolfowitz.
\newblock Optimum designs in regression problems.
\newblock {\em The Annals of Mathematical Statistics}, 30(2):271--294, 1959.

\bibitem{KoLeeWayne2}
Chun-Wa Ko, Jon Lee, and Kevin Wayne.
\newblock A spectral bound for {D}-optimality, 1994.
\newblock Manuscript \url{https://marciafampa.com/pdf/KLW1994.pdf}.

\bibitem{KoLeeWayne}
Chun-Wa Ko, Jon Lee, and Kevin Wayne.
\newblock Comparison of spectral and {H}adamard bounds for {D}-optimality.
\newblock In {\em M{ODA} 5}, pages 21--29. Physica, Heidelberg, 1998.

\bibitem{juniper}
Ole Kröger, Carleton Coffrin, Hassan Hijazi, and Harsha Nagarajan.
\newblock Juniper: An open-source nonlinear branch-and-bound solver in {J}ulia, {Proceedings of CPAIOR 2018}, 2018.

\bibitem{LeeEnv}
Jon Lee.
\newblock Maximum entropy sampling.
\newblock In A.H. El-Shaarawi and W.W. Piegorsch, editors, {\em Encyclopedia of Environmetrics, 2nd ed.}, pages 1570--1574. Wiley, Boston, 2012.

\bibitem{LeeLind2019}
Jon Lee and Joy Lind.
\newblock Generalized maximum-entropy sampling.
\newblock {\em INFOR: Information Systems and Operational Research}, 58(2):168--181, 2020.

\bibitem{li2022d}
Yongchun Li, Marcia Fampa, Jon Lee, Feng Qiu, Weijun Xie, and Rui Yao.
\newblock D-optimal data fusion: Exact and approximation algorithms.
\newblock {\em INFORMS Journal on Computing}, 2023.

\bibitem{Weijun}
Yongchun Li and Weijun Xie.
\newblock Best principal submatrix selection for the maximum entropy sampling problem: Scalable algorithms and performance guarantees.
\newblock {\em Operations Research}, 72(2):493--513, 2023.

\bibitem{MadanSinghTantXie}
Vivek Madan, Mohit Singh, Uthaipon Tantipongpipat, and Weijun Xie.
\newblock Combinatorial algorithms for optimal design.
\newblock {\em Proceedings of COLT 2019}, 99:2210--2258, 2019.

\bibitem{Appendicitis}
Ričards Marcinkevičs, Patricia {Reis Wolfertstetter}, Ugne Klimiene, Kieran Chin-Cheong, Alyssia Paschke, Julia Zerres, Markus Denzinger, David Niederberger, Sven Wellmann, Ece Ozkan, Christian Knorr, and Julia~E. Vogt.
\newblock Interpretable and intervenable ultrasonography-based machine learning models for pediatric appendicitis.
\newblock {\em Medical Image Analysis}, 91:103042, 2024.

\bibitem{Miller}
Alan~J. Miller and Nam-Ky Nguyen.
\newblock Algorithm as 295: A {F}edorov exchange algorithm for {D}-optimal design.
\newblock {\em Journal of the Royal Statistical Society. Series C (Applied Statistics)}, 43(4):669--677, 1994.

\bibitem{Nikolov}
Aleksandar Nikolov.
\newblock Randomized rounding for the largest simplex problem.
\newblock In {\em Proceedings of STOC 2015}, pages 861--870. ACM, 2015.

\bibitem{PonteFampaLeeSBPO22}
Gabriel Ponte, Marcia Fampa, and Jon Lee.
\newblock Exact and heuristic solution approaches for the {D}-optimality problem, 2022.
\newblock In: Proceedings of the LIV Brazilian Symposium on Operations Research (SBPO 2022), \url{https://proceedings.science/proceedings/100311/_papers/157447/download/fulltext_file2}.

\bibitem{pfl2024_gmesp_arxiv}
Gabriel Ponte, Marcia Fampa, and Jon Lee.
\newblock Convex relaxation for the generalized maximum-entropy sampling problem, 2024.
\newblock \url{https://arxiv.org/abs/2404.01390}.
\newblock \href {https://arxiv.org/abs/2404.01390} {\path{arXiv:2404.01390}}.

\bibitem{Puk}
Friedrich Pukelsheim.
\newblock {\em Optimal Design of Experiments}, volume~50 of {\em Classics in Applied Mathematics}.
\newblock SIAM, 2006.
\newblock Reprint of the 1993 original.

\bibitem{Sager2013}
Sebastian Sager.
\newblock {S}ampling decisions in optimum experimental design in the light of {P}ontryagin's maximum principle.
\newblock {\em {SIAM} Journal on Control and Optimization}, 51(4):3181--3207, 2013.

\bibitem{Sagnol2015}
Guillaume Sagnol and Radoslav Harman.
\newblock Computing exact d-optimal designs by mixed integer second-order cone programming.
\newblock {\em The Annals of Statistics}, 43(5):2198--2224, 2015.

\bibitem{SW}
Michael~C. Shewry and Henry~P. Wynn.
\newblock Maximum entropy sampling.
\newblock {\em Journal of Applied Statistics}, 46:165--170, 1987.

\bibitem{SinghXie_ApproxDopt}
Mohit Singh and Weijun Xie.
\newblock Approximation algorithms for {D}-optimal design.
\newblock {\em Mathematics of Operations Research}, 45(4):1512--1534, 2020.

\bibitem{Draper}
Ralph~C. St.~John and Norman~R. Draper.
\newblock D-optimality for regression designs: A review.
\newblock {\em Technometrics}, 17(1):15--23, 1975.

\bibitem{Ucinski2007}
Dariusz Uci\'nski and Maciej Patan.
\newblock D-optimal design of a monitoring network for parameter estimation of distributed systems.
\newblock {\em Journal of Global Optimization}, 39(2):291--322, 2007.

\bibitem{Boyd1998}
Lieven Vandenberghe, Stephen Boyd, and Shao-Po Wu.
\newblock Determinant maximization with linear matrix inequality constraints.
\newblock {\em SIAM Journal on Matrix Analysis and Applications}, 19(2):499--533, 1998.

\bibitem{Wald}
Abraham Wald.
\newblock On the efficient design of statistical investigations.
\newblock {\em Annals of Mathematical Statistics}, 14:134--140, 1943.

\bibitem{Welch}
William~J. Welch.
\newblock Branch-and-bound search for experimental designs based on {D}-optimality and other criteria.
\newblock {\em Technometrics}, 24(1):41--48, 1982.

\bibitem{SchurBook}
Fuzhen Zhang, editor.
\newblock {\em The {S}chur Complement and its Applications}, volume~4 of {\em Numerical Methods and Algorithms}.
\newblock Springer-Verlag, New York, 2005.

\end{thebibliography}

\section*{Appendix}

In this appendix, we present deferred proofs.

\begin{proof}[of Lem. \ref{lem:indefinite}]
Here, we denote by $\lambda_1(\cdot)\geq\lambda_2(\cdot)\geq \cdots\geq \lambda_m(\cdot)$, all the eigenvalues of an $m\times m$ given matrix. 
    We assume that
    \begin{align*}
        &\lambda_1(V^+)\geq \lambda_2(V^+)\geq \cdots\geq \lambda_{a}(V^+)>0=\lambda_{a+1}(V^+)=\cdots =\lambda_m(V^+),\\
        &\lambda_1(V^-)=\lambda_2(V^-)=\cdots=\lambda_{b}(V^-)=0>\lambda_{b+1}(V^-)\geq \cdots\geq \lambda_m(V^-).
    \end{align*}
    Then $a=r^+$ and $m-b=r^-$. Note that $a+m-b=\mbox{rank}(V)\leq m$ and, therefore, $a\leq b$. From \cite[Thm. 4.3.7]{HJBook},
    we have that 
    \[
    \lambda_{b}(V) \geq \lambda_{b}(V^-) + \lambda_{m}(V^+) = 0.
    \]
    Therefore, $\lambda_\ell(V)\geq0$, for all $\ell \leq b$. As rank$(V)=m-b+a$, then $\lambda_\ell(V) >  0$ for all $\ell \leq a$.  We also have that
    \[
    \lambda_{a+1}(V) \leq \lambda_{1}(V^-) + \lambda_{a+1}(V^+) =0.
    \]
    Therefore, $\lambda_\ell(V)\leq0$, for all $\ell \geq a+1$. As rank$(V)=m-b+a$, then $\lambda_\ell(V) < 0$ for all $\ell \geq b+1$.
    
    As $V$ has exactly $a+m-b$ nonzero eigenvalues, we conclude that $V$ has $a$ strictly positive eigenvalues and $m-b$ strictly negative eigenvalues. \qed
\end{proof}

\begin{proof}[of Prop. \ref{cor:indef}]
$V^+$ is positive semidefinite with rank at most $r^+$.
$V^-$ is negative semidefinite with rank at most $r^-$. 
Therefore, we  have \[r^++r^-=\mbox{rank}(V)\leq \mbox{rank}(V^+) + \mbox{rank}(V^-)\leq r^++r^-.
\]
 Then $\mbox{rank}(V^+)=r^+$ and $ \mbox{rank}(V^-)=r^-$. 

 From Lem. \ref{lem:indefinite}, we have that $V$ is indefinite and has $r^+$ positive eigenvalues and $r^-$ negative eigenvalues. The result then follows from Sylvester's law of inertia (see \cite[Thm. 1.5]{SchurBook} and  \cite[Thm. 4.5.8]{HJBook}). \qed
\end{proof}

\begin{proof}[of Thm. \ref{kpm}]
We note that $a>0$ because $\lambda_1>0>\lambda_2\geq \lambda_3$\,.

We see that $\bar{k}_{-}\,,\bar{k}_{+}\in\mathbb{R}$ because 
\[
\begin{array}{ll}
    b^2-4ac &= 4(\lambda_1\lambda_2+\lambda_1\lambda_3+\lambda_2\lambda_3)^2- 12\lambda_1\lambda_2\lambda_3( \lambda_1+\lambda_2+\lambda_3)\\ 
    &=4\left((\lambda_1^2\lambda_2^2+\lambda_1^2\lambda_3^2+\lambda_2^2\lambda_3^2 + 2\lambda_1^2\lambda_2\lambda_3+ 2\lambda_1\lambda_2^2\lambda_3+ 2\lambda_1\lambda_2\lambda_3^2)\right.\\
    & \quad\quad \left. - 3(\lambda_1^2\lambda_2\lambda_3 + \lambda_1\lambda_2^2\lambda_3 + \lambda_1\lambda_2\lambda_3^2)\right)\\
    &=4\left(\lambda_1^2\lambda_2^2+\lambda_1^2\lambda_3^2+\lambda_2^2\lambda_3^2 -\lambda_1^2\lambda_2\lambda_3- \lambda_1\lambda_2^2\lambda_3-\lambda_1\lambda_2\lambda_3^2\right) \\
    &=4\left( \lambda_1\lambda_2^2(\lambda_1-\lambda_3) + \lambda_1^2\lambda_3(\lambda_3-\lambda_2) +\lambda_3^2\lambda_2(\lambda_2- \lambda_1)\right) \geq 0.
\end{array}
\]
We have 
\[
    \bar{k}_{+}\geq -\frac{1}{\lambda_3}~\Leftrightarrow ~ \sqrt{b^2-4ac}\geq b-\frac{2a}{\lambda_3}.
    \]
    The last inequality is clearly satisfied  if 
$b - \frac{2a}{\lambda_3}\leq 0$. Next, we show that it is also satisfied if 
$b - \frac{2a}{\lambda_3}> 0$, when it becomes equivalent to
\begin{equation}\label{prosqrt}
\begin{array}{ll}
&b^2-4ac\geq \left(b-\frac{2a}{\lambda_3}\right)^2
\Leftrightarrow~ b\lambda_3-c\lambda_3^2-a\geq 0\\[1pt]
&\quad\Leftrightarrow~ 2\lambda_3(\lambda_1\lambda_2+\lambda_1\lambda_3+\lambda_2\lambda_3) - \lambda_3^2(\lambda_1+\lambda_2+\lambda_3) - 3\lambda_1\lambda_2\lambda_3 \geq 0\\[1pt]
&\quad \Leftrightarrow-\lambda_3^3+ \lambda_1\lambda_3^2 +\lambda_2\lambda_3^2  - \lambda_1\lambda_2\lambda_3 \geq 0\\[1pt]
&\quad\Leftrightarrow~
\lambda_3(\lambda_3-\lambda_1)(\lambda_3-\lambda_2)\geq 0.
\end{array}
\end{equation}
which follows because $\lambda_1>0>\lambda_2\geq \lambda_3$\,. We also have
\[
\bar{k}_{-}\leq -\frac{1}{\lambda_3}~\Leftrightarrow ~ -b-\sqrt{b^2-4ac}\leq -\frac{2a}{\lambda_3}~\Leftrightarrow ~ \sqrt{b^2-4ac}\geq -b+\frac{2a}{\lambda_3}. 
\]
     The last inequality  was proved in \eqref{prosqrt} for the case where  $ -b+\frac{2a}{\lambda_3}>0$. If otherwise, the inequality is clearly satisfied.

Finally,  noting that $-b+\frac{2a}{\lambda_1} = -2(\lambda_1\lambda_2+\lambda_1\lambda_3+\lambda_2\lambda_3) + 6\lambda_2\lambda_3 = -2\lambda_1\lambda_2-2\lambda_1\lambda_3+4\lambda_2\lambda_3 \geq 0$, we have
\[
 \begin{array}{ll}
    &\bar{k}_{-}\geq -\frac{1}{\lambda_1} ~\Leftrightarrow~ 
-b-\sqrt{b^2-4ac}\geq -\frac{2a}{\lambda_1} ~\Leftrightarrow~ 
b^2-4ac\leq \left(b-\frac{2a}{\lambda_1}\right)^2\\ 
&\quad \Leftrightarrow~
b\lambda_1-c\lambda_1^2-a\leq 0 \\
&\quad \Leftrightarrow~ 
2\lambda_1(\lambda_1\lambda_2+\lambda_1\lambda_3+\lambda_2\lambda_3) - \lambda_1^2(\lambda_1+\lambda_2+\lambda_3) - 3\lambda_1\lambda_2\lambda_3 \leq 0\\
&\quad \Leftrightarrow
-\lambda_1(\lambda_1-\lambda_3)(\lambda_1-\lambda_2)\leq 0,
    \end{array}
    \]
which follows because $\lambda_1>0>\lambda_2\geq \lambda_3$\,. \qed
\end{proof}

\end{document}